\pdfoutput=1
\documentclass[reqno]{amsart}
\usepackage[T1]{fontenc}

\usepackage[utf8]{inputenc}
\usepackage{mathrsfs, mathtools}
\usepackage{amsfonts, amssymb, amsmath, amsthm}
\usepackage{microtype}
\usepackage{hyperref}
\hypersetup{final, colorlinks=true, citecolor=blue}
\usepackage{url}
\usepackage{bbm}
\usepackage{enumitem}
\usepackage{xifthen}
\usepackage{tensor}
\usepackage{caption} 
\newcommand{\Q}{\mathbb{Q}}
\newcommand{\R}{\mathbb{R}}
\newcommand{\C}{\mathbb{C}}
\newcommand{\Z}{\mathbb{Z}}
\newcommand{\HH}{\mathcal{H}}
\newcommand{\cl}{\operatorname{Cl}}

\newcommand{\FA}{F_\mathbf{A}}

\newcommand{\PGL}{\operatorname{PGL}}
\newcommand{\SL}{\operatorname{SL}}
\newcommand{\Bil}{\operatorname{Bil}}

\newcommand{\val}{\operatorname{val}}

\def\ll#1{{\left\langle{#1}\right\rangle}}
\def\id#1{{\mathfrak{#1}}}
\newcommand{\idn}[2][]{\norm(\mathfrak{#2}_{#1})}
\def\hatx#1{{{\widehat{#1}}^\times}}
\newcommand{\p}{_{v}}
\newcommand{\pp}{_{\id p}}
\newcommand{\q}{_{v}}
\newcommand{\pv}{{v}}
\newcommand{\bbu}{\mathbbm{1}}
\newcommand{\ssg}{_{\gamma}}
\newcommand{\sg}{_{\gamma}}
\newcommand{\ssl}{^{l,\chi}}

\newcommand{\tlS}{\vartheta^{l,\chi}}
\newcommand{\tlu}{\tlS}
\newcommand{\etalD}{\eta\ssl_{D,\id a}}
\newcommand{\psilD}{\psi^l_D}

\newcommand{\mba}{\mathbf{a}}
\newcommand{\mbA}{\mathbf{A}}
\newcommand{\mbf}{\mathbf{f}}
\newcommand{\mbk}{\mathbf{k}}
\newcommand{\mcE}{\mathcal{E}}
\newcommand{\piv}{\varpi_{v}}
\newcommand{\pip}{\piv}
\newcommand{\piq}{\piv}
\newcommand{\set}[2]{\left\{{#1} \, : \, {#2}\right\}}
\def\tlv#1{\widetilde{\mathcal{L}}_v^{#1}}

\makeatletter
\newcommand{\tpmod}[1]{{\@displayfalse\pmod{#1}}} 
\makeatother

\newcommand{\normD}[1][D]{%
	\def\ArgI{\tlv{{#1}}}
    \normDRelay
}
\newcommand\normDRelay[1][\tphi_v]{
	\Abs{{#1}}_{\ArgI}
}

\newcommand{\tphi}{\widetilde{\varphi}}
\newcommand{\tpi}{{\widetilde{\pi}}}
\newcommand{\tpie}{\widetilde{\pi}^\epsilon}
\newcommand{\ut}[1]{\underline{\underline {#1}}}
\newcommand{\utv}[1]{\underline{\underline {{#1}_v}}}

\newcommand{\vii}{v \mid \cond{\chi}}

\def\smat#1{\left(\begin{smallmatrix} #1 \end{smallmatrix} \right)}
\def\pmat#1{\left(\begin{matrix} #1 \end{matrix} \right)}
\def\abs#1{\left\vert{#1}\right\vert}
\def\Abs#1{\left\Vert{#1}\right\Vert}
\def\num#1{\left\vert{#1}\right\vert}

\def\omegaN{\omega(\id N)}

\newcommand{\kro}[2]{\genfrac(){}{}{#1}{#2}}
\newcommand{\tkro}[2]{\genfrac(){}{1}{#1}{#2}}

\newcommand{\subp}[1][]{\ifthenelse{\equal{#1}{}}{_{\id p}}{_{\id{#1}}}}

\newcommand{\Mhku}{\Manti{4 \id{N}_\chi}}

\newcommand{\MkR}{\mathcal{M}_\mathbf{k}(R)}
\newcommand{\ckN}{c_{\mbk,\id N}}
\newcommand{\AL}[2][]{%
  \ifthenelse
    {\isempty{#1}}%
    {\varepsilon_{#2}}%
    {\varepsilon_{#2}(#1)}}

\DeclareMathOperator{\lcm}{lcm}
\DeclareMathOperator{\vardot}{\,\cdot\,}
\DeclareMathOperator{\norm}{\mathcal{N}}
\DeclareMathOperator{\trace}{\mathcal{T}}
\DeclareMathOperator{\disc}{\Delta}
\newcommand{\sgn}{\operatorname{sgn}}
\newcommand{\chic}[2][]{\chi_{#1}^{#2}}
\newcommand{\cond}[1]{\mathfrak{f}_{#1}}
\newcommand{\OO}[1][]{\mathcal{O}_{\mkern-2mu #1}}
\newcommand{\OOx}[1][]{\mathcal{O}_{\mkern-3mu #1}^{\times}}

\newcommand{\What}{\widehat{W}}
\newcommand{\hatOO}[1][]{\widehat{\mathcal{O}}_{\mkern-2mu #1}}
\newcommand{\hatOOx}[1][]{\hatx{\mathcal{O}}_{\mkern-3mu #1}}
\def\setminus{\smallsetminus}
\def\Fnc{F^\times \setminus (F^\times)^2}
\def\Fc{(F^\times)^2}

\def\dF{d_{\mkern-2mu F}}
\def\IF{\mathcal{J}_{\mkern-2mu F}}
\def\hF{h_{\mkern-2mu F}}
\def\tK{t_{\mkern-2mu K}}

\def\Tm{\mathop{T_{\mkern-1mu \id{m}}}}
\def\Tg{\mathop{T_{\mkern-2mu g}}}

\def\SN{\Sigma_{\id N}}

\numberwithin{equation}{section}

\theoremstyle{plain}
\newtheorem{prop}[equation]{Proposition}
\newtheorem{thm}[equation]{Theorem}
\newtheorem{thmA}{Theorem}

\newtheorem{thmB}{Theorem}

\newtheorem{lemma}[equation]{Lemma}

\theoremstyle{remark}

\newtheorem{rmk}[equation]{Remark}

\begin{document}

\title
[Construction of Hilbert modular forms of half-integral weight]
{Effective construction of Hilbert modular forms of half-integral
weight}

\author{Nicol\'as Sirolli}
\address{Universidad de Buenos Aires and IMAS-CONICET, Buenos Aires, Argentina}
\email{nsirolli@dm.uba.ar}

\author{Gonzalo Tornar\'ia}
\address{Universidad de la República, Montevideo, Uruguay}
\email{tornaria@cmat.edu.uy}

\begin{abstract}
	Given a Hilbert cuspidal newform $g$
    we construct a family of modular
    forms of half-integral weight whose Fourier coefficients give the
    central values of the twisted $L$-series of $g$ by
    fundamental discriminants.

    The family is parametrized by quadratic conditions on the primes
    dividing the level of $g$,
    where each form has
    coefficients supported on the discriminants satisfying the
    conditions.
	These modular forms are given as generalized theta series and thus their
	coefficients can be effectively computed.

	By considering skew-holomorphic forms of half-integral weight
	our construction works over arbitrary totally real number fields, except
	that in the case of odd degree the square levels are excluded.
	It includes all discriminants except those divisible by primes whose
	square divides the level.
\end{abstract}

\maketitle

\section{Introduction}

Let $g \in S_{2+2k}(N)$ be a normalized Hecke newform.
Answering to a question posed by Shimura in \cite{shim-halfint},
Waldspurger
showed that the
Fourier coefficients of modular forms of half-integral weight
in Shimura correspondence with $g$ are related to
the central values of the $L$-series of $g$ twisted
by quadratic characters \cite{waldspu-coeffs}.

\medskip

This result was made explicit by Kohnen and Zagier
\cite{kohnen-zagier, kohnen_fourier} 
for odd and square-free $N$.
More precisely, there exists a nonzero modular form
$f = \sum_n \lambda(n;f) q^n$
of weight $3/2+k$ and level $4N$
satisfying
\begin{equation*}
	L\left(1/2,g\otimes \chic{D}\right) = 
	\ll{g,g} \,
	\frac{\pi^{k+1}}{k!\,2^{\omega(N)}}
	\frac{1}{\abs{D}^{k+1/2}}
	\frac{\abs{\lambda(\abs{D};f)}^2}{\ll{f,f}}
\end{equation*}
for every fundamental discriminant $D$ such that
$\kro{D}{v}=\AL[v]{g}$ for all places $v$ in
$\Sigma_N=\set{v}{v\mid N}\cup\{\infty\}$.
Here $\chic{D}$ is the quadratic character induced by $D$ and
$\AL[v]{g}$ are the eigenvalues of the Atkin-Lehner involutions
acting on $g$, with
$\AL[\infty]{g}=(-1)^{k+1}$ and
$\kro{D}{\infty}=\sgn(D)$.
The $L$-series is
normalized so that it has its center of symmetry at $s=1/2$,
$\ll{g,g}$ and $\ll{f,f}$ are the Petersson inner products,
and $\omega(N)$ denotes the number of prime divisors of $N$.
Up to scalar $f$ is the unique modular form
corresponding to $g$ in the Kohnen subspace.

\medskip
Later on Baruch and Mao \cite[Thm. 10.1]{baruch-mao}
removed the quadratic restrictions on $D$. Still assuming $N$ to be odd and
square-free they show that,
given a function $\gamma:\Sigma_N\to\{\pm1\}$ with $\AL{g,\gamma}=1$
(see~\eqref{eqn:signo_de_gamma} below),
there exists a nonzero modular form $f_\gamma$ of weight $3/2+k$
satisfying
\begin{equation}\label{eqn:classic}
	L\left(1/2,g\otimes \chic{D}\right)
	=
	2^{\omega(D,N)}
	c\sg \,
	\ll{g,g} \,
	\frac{\pi^{k+1}}{k!\,2^{\omega(N)}}
	\frac1{\abs{D}^{k+1/2}}
	\frac{\abs{\lambda(\abs{D};f\sg)}^2}{\ll{f\sg,f\sg}}
\end{equation}
for every fundamental discriminant $D$ \emph{of type $\gamma$},
i.e., such that for each
$v\in\Sigma_N$ either $\kro Dv=\gamma(v)$, or $\kro Dv=0$ and
$\gamma(v)=\AL[v]{g}$.
Here $c\sg$ is an explicit positive rational constant
and $\omega(D,N)$ is the number of primes dividing both $D$ and $N$.
The modular form $f\sg$ 
{has level $4NN'$, where $N'$ is the
product of the primes in the set
$\set{v\mid N}{\gamma(v)\neq\AL[v]{g}}$.}

\bigskip

The main theorem of this article extends these results to all levels
except perfect squares,
and to Hilbert modular forms over an arbitrary totally real number
field.

Fix a totally real number field $F$ and
let $g$ be a normalized Hilbert cuspidal newform over $F$ of level
$\id N$, weight $\mathbf{2}$ + 2$\mathbf{k}$
and trivial central character.
The only hypothesis we make on the level is the following:
\begin{enumerate}[label=\textbf{H$\id N$}.,ref={\textbf{H$\id N$}},
                  labelindent=\parindent,leftmargin=*,
                  topsep=\bigskipamount
                  ]
	\item If $[F:\Q]$ is odd then $\id N$ is not a square,
		and if $\mbk\neq\mathbf0$ then $\id N$ is nontrivial.
		\label{hyp:HN}
\end{enumerate}

Let $\SN = \set{v}{v \mid \id N}\cup\mba$, where
$\mathbf a$ is the set of infinite places of $F$.
Given a function $\gamma: \SN \to \{\pm1\}$
we say that $D \in F^\times$ is \emph{of type} $\gamma$ (with respect to $g$) if
for all $v \in \SN$
\begin{equation*}
	\kro Dv =
	\begin{cases}
		\gamma(v) \text{ or } 0 &
		\text{when $\val_v(\id N) = 1$, $v$ is odd and $\gamma(v) = \AL[v]{g}$,} \\
		\gamma(v) & \text{otherwise}.
	\end{cases}
\end{equation*}
Note that for odd and square-free $\id N$ every $D \in F^\times$ is of type
$\gamma$ for some function $\gamma$.
In general this requires that
$\kro{D}{v}\neq 0$ when $\val_v(\id N)>1$
or $v\mid\gcd(\id N,2)$.

For $D$ of type $\gamma$, the sign of the functional
equation of $L(s,g\otimes\chic{D})$ is given by
\begin{equation}\label{eqn:signo_de_gamma}
	\AL{g,\gamma}
	=
	\prod_{v \in \SN}
	\AL[v]{g}\,
	\gamma(v)^{\val_v(\id N)}
	\,,
\end{equation}
where for convenience we let $\val_v(\id N)=1$ for $v\in\mba$.
In particular, if $\AL{g,\gamma} = -1$ then
$L(1/2,g\otimes\chic{D}) = 0$ for every $D$ of type $\gamma$.

\begin{thmA}\label{thm:main}
Assume $\AL{g,\gamma}=1$.
	There exists a nonzero skew-holomorphic Hilbert cuspidal form $f\sg$ of weight $\mathbf{3/2+k}$
whose Fourier coefficients $\lambda(D,\id a;f\sg)$ are effectively
computable and satisfy
\[
	L\left(1/2,g\otimes \chic{D}\right)
	=
	2^{\omega(D,\id N)} \, 
	c_{g,\gamma} \,
	\ll{g,g} \,
	\frac{\idn a}{\abs{D}^\mathbf{k+1/2}} \,
	\frac{\abs{ \lambda(D,\id a;f\sg) }^2}{\ll{f\sg,f\sg}}
\]
for every $D \in F^\times$ of type $\gamma$.
{The level of $f\sg$ is $4\id N\id N'\id N''$, where $\id N'$ is the
product of the primes in the set $\set{v\parallel\id
N}{\gamma(v)\neq\AL[v]{g}}$
and $\id N''\mid 4\OO$; in fact $\id N''$ is trivial for odd $\id N$.}
\par
Here $\id a$ is the unique ideal such that $(D,\id a)$ is a
fundamental discriminant,
$c_{g,\gamma}$ is given in $\eqref{eqn:c_gamma}$
and $\omega(D,\id N)$ denotes the number of prime ideals dividing both $\id N$
and the conductor of $F(\sqrt D)/F$.

\end{thmA}

For full generality on the field $F$
it is essential to consider skew-holomorphic
Hilbert modular forms:
see Sect.~\ref{sect:hilbert_half_integral} for the definition.
Only when $\hF=\hF^+$ it is possible to state Theorem~\ref{thm:main} using
holomorphic forms (see Remark \ref{rmk:uholo}).

\begin{rmk}

Those $D \in F^\times$ for which there is $v \mid \id N$ with $\kro Dv = 0$
such that $\val_v(\id N) > 1$ or $v$ is even are not covered by any of the
types we have considered.

When $\val_v(\id N) = 1$ we expect that our construction covers
discriminants with $\kro Dv = 0$ also in the case of even $v$.
This would follow by extending Lemma \ref{lem:edspecial} to
those places; see Remark \ref{rmk:even}.
We have verified this experimentally in some cases.

When $\val_v(\id N) > 1$ and $\kro Dv = 0$ our present framework cannot be used
since by Proposition \ref{prop:kohnen} the
fundamental discriminants on which our theta series are supported satisfy that
$\kro Dv\neq0$.
In order to cover these discriminants a different construction is needed.
\end{rmk}
\bigskip

The starting point for us \cite{gross}, where Gross 
gives the form $f\sg$ combining linearly ternary theta series,
according to the quaternionic modular form in Jacquet-Langlands
correspondence with $g$.
He considers the case of rational prime level and weight $2$, and restricts to
twists by negative discriminants (i.e., the case when $\gamma(\infty) = -1$).
His construction yields a nonzero form if and only if $L(1/2,g) \neq 0$.

The construction of Gross was extended by the authors to the setting of Hilbert
modular forms of general level in \cite{waldspu}, assuming, among other
restrictions, that $L(1/2,g) \neq 0$.
Some of these constraints were removed in \cite{w2}, by using theta series with
weight functions (of type I).
Nevertheless, the construction in \cite{w2} is still not completely general,
since it restricts the types $\gamma$ under consideration;
for example, in the case of odd $\val_v(\id N)$ only $\gamma(v) = \AL[v]{g}$ was
allowed (see Remark~\ref{rmk:gamma_rest}).
Moreover, the existence of units of arbitrary signatures was required to cover
arbitrary types at the infinite places.

In this article we remove these last restrictions introducing two new
ingredients: weight functions of \emph{type II}
and skew-holomorphic modular forms.

Weight functions were introduced, in the rational case, by \cite{mrvt} for
odd prime levels and weight $2$; see also \cite{tornatesis}.
In \cite{pt_composite} there are given examples for
odd composite levels.
They appear as well in a particular example ($N=11$) in
the final section of \cite{baruch-mao}, without a detailed discussion.
Prior to Theorem~\ref{thm:main} these generalized theta series were shown to
satisfy \eqref{eqn:classic} only in the case of rational prime levels
and weight $2$: see \cite{mao}.

Our construction of $f\sg$ is of a global nature,
and its Fourier coefficients can be effectively computed.
On the other hand, though their adelic results are general, the modular
form in \cite[Thm. 10.1]{baruch-mao} is given as a vector in an automorphic representation,
specifying its local components;
this is done only for odd and square-free levels.

\bigskip

We now sketch the construction of $f\sg$ and the proof of
Theorem~\ref{thm:main},
leaving the details for the body of the article.
Fix an order $R$ with discriminant $\id N$ and Eichler invariants
$e(R_v)\neq0$ for $v \mid \id N$, in a totally definite quaternion algebra
over $F$.
Such orders exist due to \ref{hyp:HN}.
Let $\varphi_g$ be the quaternionic modular form in
Jacquet-Langlands correspondence with $g$.

We choose an auxiliary parameter $l \in F^\times$ and a Hecke character $\chi$
such that
\begin{enumerate}[label=\textbf{H$l$}.,ref={\textbf{H$l$}},
                  labelindent=\parindent,leftmargin=*,
                  topsep=\medskipamount
                  ]
	\item $\kro lv = \gamma(v)\,e(R_v)$ for all $v\in\SN$,
		and $\kro lv \neq 0$ for all $v \mid 2$.
		\label{hyp:Hl}
\end{enumerate}
\begin{enumerate}[label=\textbf{H$\chi$}.,ref={\textbf{H$\chi$}},
				   labelindent=\parindent,leftmargin=*,
				   topsep=\medskipamount
				   ]
		   \item For finite $v$ the character $\chi_v$ is unramified if $v
			   \nmid \id N$ or $\gamma(v)^{\val_v(\id N)} = \AL[v]{g}$,
		 and is odd otherwise.
		 \label{hyp:Hchi}
\end{enumerate}
Then we define an explicit Hecke linear map $\tlS$
from the space of
quaternionic modular forms of level $R$ and weight $\mbk$ to a space of modular
forms of weight $\mathbf{3/2+k}$,
level $4\lcm(\id N,\cond\chi{}^{\!\!2})$ and character $\chi$,
whose image consists on linear combinations of
generalized theta series of the form
\begin{equation*}
	\sum_{y\in L} w(y) \, P(y)
	\, \exp_F(\disc(y)/l, z/2)\,.
\end{equation*}
Here $L$ is a ternary lattice in a quaternion space with
$\Delta(y)=(y-\overline{y})^2$ as the quadratic form,
$P$ is an (archimedean) harmonic polynomial and $w$ is a (non
archimedean) weight function depending on $l$ and $\chi$.
Finally, $\exp_F$ is the exponential map associated to $F$.

For $D \in F^\times$ denote by
$L_{l,D}(s,g) = L(s,g\otimes \chic{l}) \,L(s,g\otimes \chic{D})$ 
the Rankin--Selberg convolution \textit{L}-function of $g$ by the genus character
associated to the pair $(l,D)$.
\begin{thmB}\label{thm:formula}
If $l$ and $\chi$ satisfy \ref{hyp:Hl} and \ref{hyp:Hchi} then
\begin{equation*}
	L_{l,D}(1/2,g)
	=
	2^{\omega(D,\id N)} \,
	\ckN \,
	\ll{g,g}
	\,
	\frac{\idn{a/b}}{\abs{lD}^\mathbf{k+1/2}}
	\,
	\frac{\abs{ \lambda(D,\id a;\tlS(\varphi_g))}^2}
	     {\ll{\varphi_g,\varphi_g}}
\end{equation*}
for every $D \in F^\times$ of type $\gamma$.
\par
Here $\id a,\id b$
are the unique ideals such that $(D,\id a)$ and $(l,\id b)$ are
fundamental discriminants
and $\ckN$ is a positive constant given in \eqref{eqn:ckn}.
\end{thmB}

Theorem \ref{thm:main} follows from Theorem \ref{thm:formula}
by choosing $l$ and $\chi$ with $L(1/2,g\otimes\chi^l) \neq 0$ and
$\chi$ of minimal conductor, and letting
$f\ssg = \tlu(\varphi_g)$ (see Sect.~\ref{sect:formula} for the details).
The proof of Theorem~\ref{thm:formula} relies on our previous articles
\cite{waldspu,w2}.
It is based on the results from
\cite{zhang-gl2,Xue-Rankin} giving central values in terms of height parings,
which we relate to Fourier coefficients.

\bigskip

This article is organized as follows.
In Sects.~\ref{sect:hilbert_half_integral} and \ref{sect:quat_forms}
we recall the basic definitions and facts regarding 
skew-holomorphic Hilbert modular forms of half-integral weight
and quaternionic modular forms.
In Sect.~\ref{sect:construction} we introduce the weight functions
and use them to define our generalized theta series.
Then in Sect.~\ref{sect:modularity} we prove the modularity of this
construction and its Hecke linearity.
In Sect.~\ref{sect:formula} we give a proof of Theorem \ref{thm:formula}.
This proof, in principle, imposes certain restrictions on the
discriminants.
The goal of Sects.~\ref{sect:shimura} and \ref{sect:extending} is to prove
Proposition \ref{prop:bmw}, which is used to remove these restrictions.
For this it suffices to prove, \emph{a posteriori}, that our construction
agrees locally
with the vectors considered by Baruch and Mao
for $v\nmid\id N$ and for $v\parallel\id N$ and $v$ odd.

In Sect.~\ref{sect:Q} we describe our results in the case $F = \Q$.
Finally, using Sagemath \cite{sage} and Magma \cite{magma}
we illustrate our results by computing the families of modular forms $f_\gamma$ in
many different situations.
These can be found in \cite{code}.
One of these examples is described in Sect.~\ref{sect:example}.

\subsection{Notation summary}

We fix a totally real number field $F$ of discriminant $\dF$
with ring of integers $\OO$ and different ideal $\id{d}$.
We denote by $\FA$ the ring of adeles of $F$ and
by $\IF$ the group of fractional ideals of $F$. We write $\cl(F)$
for the class group, $C_F$ for the idele class group and $\hF$ for the class 
number.
We denote by $\mba$ the set of embeddings $v : F \hookrightarrow \R$
and by $\mathbf{f}$ the set of nonzero prime ideals $v$ of $F$.
For $\xi \in F^\times$ we let $\sgn(\xi) = \prod_{v \in \mba} \sgn(\xi_v)$.
We denote $F^+ = \set{\xi \in F^\times}{\sgn(\xi_v) = 1\quad\forall v\,\in \mba}$.
Given $\mathbf{k} = (k_v)\in \Z^\mba$ and $\xi \in F$
we let $\xi^\mathbf{k} = \prod_{v\in \mba} \xi_v^{k_v}$
and
$\abs{\xi}^\mathbf{k+1/2} = \prod_{v\in \mba} \abs{\xi_v}^{k_v+1/2}$.
We use $\pv$
as a subindex to denote completions of global objects at $\pv$, as well as to
denote local objects.
Given $v \in \mbf$ we denote by
$\pip\in\OO[v]$ a local uniformizer and by
$\val_v$ the $v$-adic valuation.

We denote by $\widehat F = \prod'_{v \in \mbf} F\p$ and
$\hatx F = \prod'_{v \in \mbf} F\p^\times$ the corresponding restricted
products, and we use the same notation in other contexts.
For convenience, given $\xi \in \hatx F$ we denote
$\xi \OO  = (\xi \hatOO) \cap F \in \IF$.

For $z \in \C$ we denote $\exp(z) = e^{2\pi iz}$.
For each place $v$ we let $\exp_v:F_v \to \C$,
\begin{equation*}
	\exp_v(x_v) =
	\begin{cases}
		\exp(x_v) & \text{if $v \in \mba$,} \\
		\exp(-\mathop{\mathrm tr}_{F_v/\Q_p}(x_v))
				   & \text{if $v \in \mbf.$}
	\end{cases}
\end{equation*}
Then we define the exponential map on $F_\mbA$ by
$\exp_\mbA(x) = \prod_v \exp_v(x_v)$ and we denote its restriction to 
$\widehat F$ by $\exp_\mbf$.
Finally, denoting by $\HH$ the complex upper half-plane,
we let
\[
	\exp_F : F\times \HH^\mathbf{a} \to \C, \quad (\xi, z) \mapsto
	\prod_{\substack{v \in \mba ,\, \xi_v>0}}
	\exp(\xi_v\, z_v)
	\prod_{\substack{v \in \mba ,\, \xi_v<0}}
	\exp(\xi_v\, \overline{z_v})
	\,.
\]
For $F=\Q$ we have $\exp_\Q(\xi,z)=q^\xi$ when $\xi>0$
and $\exp_\Q(\xi,z)=\overline{q}^\xi$ when $\xi<0$.

Given a character $\chi$ of $C_F$ of conductor $\cond\chi$ and
$\id c\subseteq\OO$ we denote
\[
 \chi_\mathbf{a} = \prod_{v \in \mathbf{a}} \chi_v
 \,,\qquad
 \chi_{\id c} = \prod_{v \mid \id{c}} \chi\p
 \qquad\text{and}\qquad
 \chi_* = \prod_{\substack{v\in\mbf,\,v\nmid\cond\chi}}\chi\p
 \,. 
\]
We also use $\chi_*$ for the character induced on ideals prime to
$\cond\chi$.
Given $\xi \in 
F^\times$ we denote by the $\chic{\xi}$ the character of 
$C_F$ corresponding to the extension $F(\sqrt{\xi})/F$, and we denote 
its conductor by $\cond{\xi}$. 

Given a quadratic extension $K/F$ we let $\OO[K]$ be the maximal order, and
when it is totally imaginary we let $\tK = [\OOx[K]:\OOx]$.
For $x\in K$ we let $\disc(x)=(x-\overline{x})^2$.
If $K = F(\sqrt \xi)$ with $\xi \in \Fnc$ and $\id a \in \IF$ we say that
the pair $(\xi,\id a)$ is a \emph{discriminant} if there exists $\omega \in K$
with $\disc(\omega) = \xi$ such that $\OO \oplus \id a\,\omega$ is an order in $K$.
When this order equals $\OO[K]$
we say that the discriminant $(\xi,\id a)$ is \emph{fundamental}. 
For each $\xi$ there exists a unique $\id a$ such that $(\xi,\id a)$ is
fundamental (see \cite[Prop.~2.11]{waldspu}).
Finally, for each place $v$ of $F$ we let
\[
	\kro\xi v =
	\begin{cases}
		 1 & \text{if $v$ remains inert in $K$,} \\
		-1 & \text{if $v$ splits in $K$}, \\
		 0 & \text{if $v$ ramifies in $K$.}
	\end{cases}
\]
For completeness, for $\xi \in F^\times$ we say that $(\xi^2,\xi^{-1} \OO)$ is a
fundamental discriminant and we let $\kro{\xi^2}v = 1$ for all $v$.

Given a quaternion algebra $B/F$ we denote by $\norm:B^\times\to F^\times$ and
$\trace:B\to F$ the reduced norm and trace maps, and we use $\norm$ and $\trace$
to denote other norms and traces as well. 
In particular $\Delta=\trace^2-4\norm$.


\subsection{Acknowledgments}

We would like to thank the anonymous referees for their valuable suggestions,
which helped us improve the original manuscript.
We also thank Ariel Pacetti and Zhengyu Mao for the many conversations we shared
on this subject along the years.

\section{Half-integral weight Hilbert modular forms}
\label{sect:hilbert_half_integral}

\def\Ga#1{\Gamma_{#1}}
\newcommand{\ah}{\rho}
\def\Manti#1{\mathcal M_{\mathbf{3/2+k}}^\ah(#1,\chi)}
\def\Mholo#1{\mathcal M_{\mathbf{3/2+k}}(#1,\chi)}
\newcommand{\MA}{\widetilde{\SL_2}(\FA)}

We consider \emph{skew-holomorphic} forms,
i.e., half-integral weight Hilbert modular forms
that are holomorphic in
a subset of places of $\mba$ and antiholomorphic in its complement.
The holomorphic case is considered in \cite[Sect.~3]{shim-hh}.
Here we summarize the basic definitions and results in the
skew-holomorphic case,
some of which are implicit in \cite[Sect.~11]{shim-hh}.

\medskip

Denoting by $\HH$ the complex upper half-plane, we say that a function $f$ on
$\HH$ is \emph{antiholomorphic} if $z \mapsto f(-\overline z)$ is holomorphic.

For an integral ideal $\id c$ divisible by $4$, following Shimura, we denote
\begin{equation*}
	\Gamma_{\id c}
	=
	\set{\beta = \pmat{a_\beta & b_\beta \\ c_\beta & d_\beta} \in \SL_2(F)}
	{a_\beta,d_\beta \in \OO, b_\beta \in 2 \id{d}^{-1}, c_\beta \in 2^{-1}
		\id c \id d}.
\end{equation*}
		
Let $\mathbf{k} \in \Z^\mba$ and let $\ah \in \{\pm1\}^\mba$.
We consider the automorphy factor of weight $\mathbf{3/2+k}$ and
signature $\ah$ defined
for $\beta \in \Ga{4\OO}$ and $z \in \HH^\mba$ by
\begin{align*}
	K_\ah(\beta,z)
	=
	h(\beta,z)
	\prod_{\ah_v = 1} j(\beta_v,z_v)^{k_v+1}
	\prod_{\ah_v = -1} j(\beta_v,z_v)^{-1} \abs{j(\beta_v,z_v)}
	j(\beta,\overline{z_v})^{k_v+1}
	\, .
\end{align*}
Here $j$ and $h$ denote respectively the standard weight $1$ and weight
$\mathbf{1/2}$ automorphy factors.

Fix an ideal $\id c$ divisible by $4$ and
let $\chi$ be a character of $C_F$ with conductor dividing $\id c$.
We denote by $\Manti{\id c}$ the space of functions $f : \HH^\mba \to \C$
that are holomorphic in the variables $v$ such that $\ah_v = 1$ and
antiholomorphic in the remaining variables, satisfying
\begin{equation}\label{eqn:modularity}
	f(\beta z)
	=
	\chi_{\id c}(a_\beta) \,
	K_\ah(\beta,z)\,
	f(z)
\end{equation}
for every $\beta \in \Ga{\id c}$.
Furthermore, we require the usual regularity at the cusps if $F = \Q$.

When $\ah_v = 1$ for every $v$ we obtain the usual space of
holomorphic modular forms, which we denote by $\Mholo{\id c}$.
For $F = \Q$ we have that $f(z) \mapsto f(2z)$ gives an isomorphism between
$\Mholo{\id c}$ and the classical space of modular forms of level $c = \idn{c}$
and character $\chi$ (respectively $\chi\,\chi^{-1}$) when $k$ is odd
(respectively even).
The homothety by $2$ is explained by the difference between the standard
$\Gamma_0(c)$ and $\Gamma_{\id c}$ as defined above.
The change in the character appears because in the classical setting the
automorphy factor is given by $h^{3+2k}$ instead of $hj^{k+1}$.

\begin{rmk}\label{rmk:mholozero}
The space $\Manti{\id c}$ is zero unless
\begin{equation*}
	\chi_\mba(-1) \prod_{v\in\mba}\ah_v = (-1)^{\mathbf{k+1}}.
\end{equation*}
Furthermore, this space depends only on the restriction of $\chi\p$
to $\OOx[v]$ for $\vii$,
and not on the other local components of $\chi$.
\end{rmk}

\medskip

We consider the action of $F^\times$ on $\HH^\mba$ given by
$u \cdot z = w$, where $w_v = u_v z_v$ if $u_v > 0$ and 
$w_v = u_v \overline{z_v}$ if $u_v < 0$.
We denote by $(f,u) \mapsto f \cdot u$ the action it induces on functions in
$\HH^\mba$.

\begin{prop}\label{prop:hilbert_holomorphic}

Assume that $u \in \OO$ is such that $\sgn(u_v) = \ah_v$ for every $v \in
\mba$.
Then we have an injection
\begin{equation*}
	\Manti{\id c} \longrightarrow
	\mathcal M_{\mathbf{3/2+k}}\left(u \id c,\chi \chi^{u}\right),
	\qquad
	f \mapsto f \cdot u \,,
\end{equation*}
which is an isomorphism if $u \in \OOx$.

\begin{proof}

The second claim is straightforward.
The first claim is proved by replacing $z$ by $u\cdot z $ in
\eqref{eqn:modularity} and using that
\begin{equation}\label{eqn:shim33}
	K_\ah(\beta,u\cdot z) =
    \chi^u_{\cond u}(a_\beta) \,
	K\left( \beta^u, z\right)
    \qquad\forall\beta\in\Ga{4u\OO}
    \,,
\end{equation}
where $K$ is the automorphy factor for holomorphic forms and
$\beta^u = \smat{u&0\\0&1}\beta\smat{u&0\\0&1}^{\scriptscriptstyle-1}$.

To prove \eqref{eqn:shim33}, we note that since
$j(\beta_v^{u_v},z_v) = j(\beta_v,u_v z_v)$ we have that
\begin{equation*}
	\frac{K_\ah(\beta,u\cdot z)}
	{K\left( \beta^u,z\right)}
	=
	\frac{h(\beta,u\cdot z)}
	{h(\beta^u,z)} \,
	\prod_{u_v < 0}
	\frac{\abs{j_v(\beta_v,u_v \overline{z_v})}}
	{j_v(\beta_v,u_v \overline{z_v})}.
\end{equation*}
Then the result follows from \cite[Prop.~11.3]{shim-hh},
which proves that for $u \in \OO$ and $\beta \in \Ga{4u\OO}$
the right hand side of this equation is equal to
$\chi^u_{\cond u}(a_\beta)$.
\end{proof}

\end{prop}

\begin{rmk}\label{rmk:uholo}

In particular, when there exist units of arbitrary signature,
i.e., when $h_F = h^+_F$,
we have that every space $\Manti{\id c}$ is isomorphic to a
space of holomorphic forms with the same level.

\end{rmk}

\medskip

Let $f \in \Manti{\id{c}}$.
The automorphic form induced by $f$ is the function $f_\mbA$ on $\MA$ given by
\begin{equation*}
	f_\mbA(\beta\alpha) =
	\chi_{\id c}(a_\beta)^{-1}
	K_\ah(\alpha,\mathbf i)^{-1}
	f(\alpha \, \mathbf i),
\end{equation*}
where $\beta \in \SL_2(F)$ and $\alpha \in \MA$ are such that
$\alpha_\mbf$ belongs to the adelized $\Ga{\id c}$.

Let $\id a \in \IF$, and let $t \in \widehat F^\times$ be such that
$t \OO = \id a$.
Then the Fourier coefficients $\lambda(\xi, \id a;f)$ for $\xi \in F$ by
are defined in terms of $f_\mbA$ by the equality
\begin{equation}\label{eqn:fourier}
	\chi_\mbf(t) \,
	\idn a^{-1/2}
	f_\mbA\Bigl(\smat{t & s \\ 0 & t^{-1}}\Bigr)
	=
	\sum_{\xi \in F}
	\lambda(\xi, \id a;f) \,
	\exp_F(\xi,(s+ \mathbf i) /2)
\end{equation}
for all $s\in F_\mba$,
where $\lambda(\xi, \id a;f)=0$ unless $\xi_v\ah_v \geq 0$ for
every $v \in \mba$.
In particular, for $u$ as in Proposition \ref{prop:hilbert_holomorphic}
they satisfy that
\begin{equation}\label{eqn:fu}
	\lambda\left(u \xi, \id a; f\cdot u\right)
	=
	\lambda(\xi, \id a; f)
	\qquad
	\forall\,\xi \in F,\,
	\id a \in \IF.
\end{equation}
Moreover, we have that
\begin{equation*}
	f(z)
	=
	\sum_{\xi \in F}
	\lambda(\xi,\OO;f) \,
	\exp_F(\xi,z /2) \,.
\end{equation*}


\section{Quaternionic modular forms}\label{sect:quat_forms}

In this section we recall 
the basic definitions and facts regarding
quaternionic modular forms.
We follow \cite[Sect.~1]{waldspu} and consider only the case of
totally definite quaternion algebras.
In this case quaternionic modular forms can be defined as functions
on the \emph{finite} ideles.

\medskip

Let $B$ be a totally definite quaternion algebra over $F$.
We let $W = B/F$, in which we consider the totally negative definite
ternary quadratic form $\Delta(x)=(x-\overline x)^2$.
We consider the space $V_\mathbf{k}$ of homogeneous polynomials in $W$ of
degree $\mathbf{k}$, harmonic with respect to $\Delta$.
We denote by $(P,\gamma)\mapsto P\cdot \gamma$ the action of 
$B^\times /F^\times$ on $V_\mathbf{k}$ by conjugation.

Let $R$ be an order
in $B$. A \emph{quaternionic modular form} of weight 
${\mathbf{k}}$ and level $R$ is a function $\varphi: \hatx B\to
V_{\mathbf{k}}$ such that for every 
$x\in \hatx B$ the following transformation formula is satisfied:
\[
 \varphi(u x \gamma) = \varphi(x) \cdot \gamma \qquad
 \forall\,u\in\hatx R,\,\gamma\in B^\times\,.
\]
The space of all such functions is denoted by $\mathcal{M}_{\mathbf{k}}(R)$. We let 
$\mathcal{E}_{\mathbf{k}}(R)$ be the subspace of functions that factor through the map 
$\norm:\hatx B \to \hatx F$. These spaces come equipped with the action 
of Hecke operators $\Tm$, indexed by integral ideals $\id m \subseteq \OO$,
and given by
\begin{equation}\label{eqn:quat_hecke}
\Tm\varphi(x) = \sum_{h \in \hatx R \backslash H_\id{m}} \varphi(h x) 
\,,
\end{equation}
where $H_\id{m} = \left\{h \in \widehat R \, : \, \norm(h) \, \OO = \id m\right\}$.

Given $x \in \hatx B$ we let
\[
 \widehat R_x = x ^{-1} \widehat R \, x \,,
 \qquad R_x = B\cap \widehat R_x \,, 
 \qquad \Gamma_{\!x} = R_x^\times / \OOx,
 \qquad t_x = \num{ \Gamma_{\!x} }\,.
\]
The sets $\Gamma_{\!x}$ are finite since $B$ is totally definite. Let $\cl(R) = 
\hatx R \backslash \hatx B / B^\times$. We define an inner product on 
$\mathcal{M}_{\mathbf{k}}(R)$, called the \emph{height pairing}, by
\[
 \ll{\varphi,\psi} = \sum_{x\in \cl(R)} \tfrac{1}{t_x}
 \ll{\varphi(x),\psi(x)}\,. 
\]
The space of \emph{cuspidal forms} $\mathcal{S}_{\mathbf{k}}(R)$ is defined as the 
orthogonal complement of $\mathcal{E}_{\mathbf{k}}(R)$ with respect to
the height pairing.

Denote $\mathcal{M}_{\mathbf{k}}(R,\bbu)$ and 
$\mathcal{S}_{\mathbf{k}}(R,\bbu)$ the subspaces
of quaternionic modular forms invariant by $\hatx F$.
Let $N(\widehat R) = \{z \in \hatx B\,:\,\widehat R_z = \widehat R\}$ be the 
normalizer of $\widehat R$ in $\hatx B$
and let $\Bil(R) = \hatx R \backslash N(\widehat R) / 
\hatx F$.
The group $\Bil(R)$ acts on $\mathcal{M}_{\mathbf{k}}(R,\bbu)$ and 
$\mathcal{S}_{\mathbf{k}}(R,\bbu)$
by $(\varphi\cdot z)(x)=\varphi (zx)$,
and this action is
related to the height pairing by the equality 
\begin{equation}\notag
\ll{\varphi \cdot z,\psi \cdot z} = \ll{\varphi, \psi}\,.
\end{equation}
Given a character $\delta$ of $\Bil(R)$ we denote
\[
  \mathcal{M}_{\mathbf{k}}(R,\bbu)^\delta =
  \left\{ \varphi \in \mathcal{M}_{\mathbf{k}}(R,\bbu) \; : \;
  \varphi \cdot z = \delta(z) \, \varphi \quad \forall z \in \Bil(R)\right\}.
\]

Given $x\in \hatx B$ and $P\in V_{\mathbf{k}}$, let $\varphi_{x,P} \in\mathcal{M} 
_{\mathbf{k}}(R)$ be the quaternionic modular form given by
\begin{equation}\label{eqn:qmf}
 \varphi_{x,P}(y) = \sum_{\gamma\in \Gamma_{\!x,y}} P \cdot \gamma\,,
\end{equation}
where $\Gamma_{\!x, y} = (B^\times \cap x ^{-1} \hatx R y) / \OOx$. 
Note that $\varphi_{x,P}$ is supported in $\hatx R x B^\times$.

Given $\varphi \in \mathcal{M}_{\mathbf{k}}(R)$, using that $\varphi(x) \in 
V^{\Gamma_{\!x}}$ for every $x \in \hatx B$ we get that
\begin{equation}\label{eqn:phi=sum_phix}
 \varphi = \sum_{x\in \cl(R)} \tfrac{1}{t_x} \varphi_{x,\varphi(x)}\,.
\end{equation}

\subsection{Locally residually unramified orders}

Denote by $\id N$ the discriminant of $R$.
We say that $R$ is \emph{locally residually unramified} 
if the Eichler invariant $e(R_v)$ is nonzero for every $v \mid \id N$.

The local classification of such orders is well known.
If $v\nmid\id N$ or if $e(R_v) = 1$ then
\begin{equation}\label{eqn:e=1}
 R\p \simeq \left\{\left(\begin{smallmatrix} a &  b \\ \pip^r c & d
\end{smallmatrix}\right) \,:\, a,b,c,d \in \OO[\pv] \right\},
\end{equation}
where $r = \val\p(\id N)$. In this case $B_v$ is always split
and the discriminant form is $\Delta=(a-d)^2 + 4\pip^r bc$.

If $e(R\p) = -1$ let $E\p$ be the unique
unramified quadratic extension of $F\p$. Then
\begin{equation}\label{eqn:e=-1}
 R\p \simeq \left\{\left(\begin{smallmatrix} \alpha  & \pip^ r \beta \\ \pip^{r+t}
\overline{\beta} & \overline{\alpha} \end{smallmatrix}\right) \,:\, \alpha,\beta \in \OO[E\p]
\right\},
\end{equation}
where $t\in \{0,1\}$ and $2r+t = \val\p(\id N)$. In this case $B_v$ is
split when $t=0$ and ramified when $t=1$.
The discriminant form is
$\Delta=\Delta(\alpha)+4\pip^{2r+t}\norm{\beta}$.
\begin{prop}
    \label{prop:disc_eichler}
    Let $v\mid\id N$ such that $e(R_v)\neq0$,
    and let $x\in R_v$.
    \begin{enumerate}
        \item
            If $\Delta(x)\in\OOx[v]$, then $\kro{\Delta(x)}{v}=e(R_v)$.
        \item
            If $\val_v(\id N) > 1$ and
            $\OO[v][x]$ is a maximal order
            then $\Delta(x)\in\OOx[v]$.
    \end{enumerate}
\end{prop}
\begin{proof}
    Let $\delta\equiv1\pmod{4\OO[v]}$ such that $\kro{\delta}{v}=e(R_v)$.
    From \eqref{eqn:e=1} and \eqref{eqn:e=-1} it follows that
    $\Delta(x)\equiv \delta t^2\pmod{4\id N}$
    with $t\in\OO[v]$. In particular
    $t\equiv\trace(x)\pmod{2\OO[v]}$.
    If $\Delta(x)\in\OOx[v]$ then
    $\kro{\Delta(x)}{v}=\kro{\delta}{v}=e(R_v)$,
    proving (a).
    \par
    Now suppose $\pip^2\mid\id N$
    and $\pip\mid\Delta(x)$.
    Then
    $\pip\mid t$ and
    $\Delta(x)\equiv \delta t^2\equiv t^2\pmod{4\pip^2}$.
    Denoting $y=x+\frac{t-\trace(x)}2$
    it follows that $\trace(y)=t$ and
	$\norm(y)=\frac{t^2-\Delta(x)}4$,
    hence $\pip\mid\trace(y)$ and
	$\pip^2\mid\norm(y)$
    so that $y/\pip$ is integral
	and $\OO[v][x]=\OO[v][y]$ is not maximal,
    proving (b).
\end{proof}

We recall
that for these orders
$\Bil(R)=\prod_{v\mid\id N} \Bil(R\p)$, where
$\Bil(R\p)$ is cyclic of order $2$ generated by the matrix $z_v$
given, in terms of the above description, by
\begin{equation}\label{eqn:bilgen}
	z_v =
	\begin{cases}
		\smat{0&1\\\pip^r&0}
			   & \text{if $e(R\p)=1$,}
			   \\[3pt] 
		\smat{0&\pip^r\\\pip^{r+t}&0}
			   & \text{if $e(R\p)=-1$.}
	\end{cases}
\end{equation}
We denote by $\iota\p$ the unique isomorphism
$\iota\p : \Bil(R\p) \to \{\pm 1\}$.


\section{Weight functions and theta series}\label{sect:construction}

Given an ideal $\id N$, from now on we let $R$ be a fixed locally residually
unramified order with discriminant $\id N$ in a totally definite quaternion
algebra $B$ over $F$.
There is at least one choice for $R$ except when
$[F:\Q]$ is odd and $\id N$ is a square.
For instance, when $[F:\Q]$ is even, one choice is to take $B$ to be the
quaternion algebra ramified at the infinite places and $R$ an Eichler
order, i.e.,  with $e(R_v)=1$ for all $v \mid\id N$.
On the other hand, when $[F:\Q]$ is odd, the quaternion algebra must
ramify in at least one finite prime: this is
possible precisely when $\val_v(\id N)$ is odd for some $v\mid\id N$,
i.e., when $\id N$ is not a square.

\medskip

We fix $l \in F^\times$ such that
$\cond l$ is prime to $2 \id N$.
In particular, $\cond l$ is square-free.
We also fix a Hecke character $\chi$ unramified outside $\SN$ such that
for all $v \mid \cond\chi$ we have that
$\chi_v$ is an odd character.
In practice we can take $\chi_v$ of the smallest possible conductor, namely
$2\pip \OO[v]$.

We denote by $\delta\ssl$ be the character of $\Bil(R)$ given by
\begin{equation}\label{eqn:deltal}
	\delta\ssl(z) = \chic[*]{l}(\norm z) \, 
	\prod_{\vii} \iota\p(z\p)
	\,.
\end{equation}

\subsection{Local weight functions}

We consider the lattice $L\p = R\p / \OO[v] \subseteq W\p$
with the quadratic form $\Delta$. The associated bilinear form is
given by $\ll{x,y} = \Delta(x+y) - \Delta(x) - \Delta(y)$, so that
$\Delta(x)=\ll{x,x}/2$.
When using the identifications from $\eqref{eqn:e=1}$ and $\eqref{eqn:e=-1}$ we
write elements of $L\p$ as $2\times2$
matrices, which should be understood modulo $\OO[v]$.

\subsubsection*{Type I}

For each prime $v \mid \cond l$ we fix $z \in L\q \setminus \piq L\q$ such
that $\pip\mid\Delta(z)$.
Such $z$ exists by \eqref{eqn:e=1}.
Since $v\nmid 2\id N$ the lattice $L_v$ is unimodular.
It follows that if $x \in L\q \setminus \piq L\q$ is such
that $\pip\mid\Delta(x)$ and $\pip\mid\ll{x,z}$, then
there exists $\xi \in \OOx[v]$ such that $x\equiv\xi z\pmod{\piq L\q}$.
We let $w\q$ be the (local) \emph{weight function} on $L\q$ given by
\begin{equation*}
	w\q(x) =
	\begin{cases}
		0 & \text{if $\pip\nmid\Delta(x)$ or $x\equiv0\tpmod{\piq L\q}$,} \\
		\chic[v]{l}(-\ll{x,z})
		   & \text{if $\pip\mid\Delta(x)$ and $\pip\nmid\ll{x,z}$,} \\
		\chic[v]{l}(\xi)
		   & \text{if $\pip\mid\Delta(x)$ and $x\equiv\xi z\tpmod{\piq L\q}$.}
	\end{cases}
\end{equation*}

Weight functions of type I are already considered in \cite{w2},
where the following properties are stated without proof.

\begin{prop}\label{prop:wq}
	The function $w\q$ is $\piq L\q$-periodic and satisfies
\begin{alignat}{2}
    w\q (\xi x) & = \chic[v]{l}(\xi) \, w\p(x) &&
	\forall\,\xi \in \OO[v]^\times,
	\label{eqn:wqhomog}\\
    w\q(y x y^{-1}) & = \chic[v]{l}(\norm y) \, w\q(x) \qquad &&
	\forall\,y \in R\q^\times.
	\label{eqn:wqtransp}
\end{alignat}
\end{prop}

\begin{proof}
	Let $x,x'\in L\p$ such that $x\equiv x'\pmod{\pip L\p}$.
	Then $\Delta(x)\equiv\Delta(x')\pmod\pip$
	and $\ll{x,z}\equiv\ll{x',z}\pmod\pip$,
	proving the first claim.
	\par
	The equality~\eqref{eqn:wqhomog} is clear.
	To prove~\eqref{eqn:wqtransp} we use the identification
	$R\p=M_2(\OO[v])$.
	Let $x=\smat{a&b\\c&d}$ and assume
	$\pip\mid\Delta(x)$,
	since $w\p(x)=w\p(yxy^{-1})=0$ otherwise.
	\par
	Consider first $z=z_0=\smat{0&0\\1&0}$.
	A simple calculation shows that in this case $w\p(x)=\chic[v]{l}(c)$ when
	$c\not\equiv0\pmod\pip$ and $w\p(x)=\chic[v]{l}(-b)$ when
	$b\not\equiv0\pmod\pip$.
	Now let $y=\smat{r&s\\t&u}\in R\p^\times$ so that
	$\overline{y}=\smat{u&-s\\-t&r}$.
	If $c\equiv0\pmod\pip$ then $a-d\equiv0\pmod\pip$ and
	\[
		yx\overline{y}\equiv\smat{\ast&br^2\\-bt^2&\ast}\pmod{\pip L\p},
	\]
	hence $w\p(yx\overline{y})=\chic[v]{l}(-b)=w(x)$.
	If $c\not\equiv0\pmod\pip$ then
	\[
		yx\overline{y}\equiv\smat{\ast&-c\left(s+\frac{(a-d)r}{2c}\right)^2\\
	c\left(u+\frac{(a-d)t}{2c}\right)^2&\ast}\pmod{\pip L\p},
	\]
	hence $w\p(yx\overline{y})=\chic[v]{l}(c)=w(x)$.
	Since $\overline{y}=\norm(y) y^{-1}$, we
	obtain \eqref{eqn:wqtransp} from \eqref{eqn:wqhomog}.
	\par
	It remains to prove~\eqref{eqn:wqtransp} for arbitrary $z$.
	Say $z'\equiv\xi(y^{-1}z_0y)$ with $\xi\in\OOx[v]$, $y\in
	R\p^\times$,
	and let $w'_v$ be the corresponding weight function.
	It is easy to see that $w'_v(x)=w_v(\xi(yxy^{-1}))$.
	Using \eqref{eqn:wphomog} and~\eqref{eqn:wptransp} for $w_v$
	we conclude that $w'_v$ is a constant multiple of $w_v$;
	therefore it also satisfies \eqref{eqn:wptransp}.
\end{proof}

\subsubsection*{Type II}
For each prime $\vii$
we fix $z \in L\p$ such that
$\Delta(z) \in  \OOx[v]$.
We let $w\p$ be the (local) \emph{weight function} on $L\p$ given by
\begin{equation}\label{eqn:typeII}
	w\p(x) =
	\begin{cases}
		\overline{\chi\p}\bigl(\ll{x,z}/\ll{z,z}\bigr)
		   & \text{if $\ll{x,z}/\ll{z,z} \in\OOx[v]$,} \\
		0  & \text{otherwise.}
	\end{cases}
\end{equation}
Using \eqref{eqn:e=1} and \eqref{eqn:e=-1} with $v\mid\id N$
it is easy to see that
\begin{equation}
	\label{eqn:sop_wii}
	w\p(x) \neq 0 
	\quad
	\Longleftrightarrow
	\quad
	\Delta(x)\in\OOx[v].
\end{equation}

We consider the lattice
\[
	L\p^0
	= \set{x \in L\p}{\ll{x,L\p} \subseteq 4 \pip \OO[v]}.
\]

\begin{prop}\label{prop:wp}
The function $w\p$ is $L\p^0$-periodic and satisfies
\begin{alignat}{2}
    w\p (\xi x) & = \overline{\chi\p}(\xi) \, w\p(x) &&
    \forall\,\xi \in \OO[\pv]^\times,
	\label{eqn:wphomog}\\
    w\p(y x y^{-1}) & = \,\iota\p(y) \, w\p(x) \qquad &&
	\forall\,y \in \Bil(R\p).
	\label{eqn:wptransp}
\end{alignat}
\end{prop}

\begin{proof}

	Let $x,x'\in L\p$ such that $x\equiv x'\pmod{L\p^0}$.
	Then $\ll{x,z} \equiv \ll{x',z} \pmod{4\pip}$
	implies
	$\ll{x,z}/\ll{z,z} \equiv \ll{x',z}/\ll{z,z} \pmod{2\pip}$,
	which proves the first claim.

	The equality \eqref{eqn:wphomog} is clear.
    Since $\chi\p$ is odd, to prove \eqref{eqn:wptransp}
	it suffices to show that (modulo $L\p^0$) we have
	\[
		y x y^{-1} \equiv
		\begin{cases}
			 x & \text{if $y \in R\p^\times$,} \\
			-x & \text{if $y \notin R\p^\times$.}
		\end{cases}
	\]

	We consider the case $e(R\p) = 1$ with the identification for
	$R\p$ given by
	\eqref{eqn:e=1}.
	Then
	\[
		L\p^0 =
		\set{\smat{\pip a & b \\ \pip^r c & -\pip a}}{a,b,c \in \OO[v]}.
	\]
	In particular
	$x\equiv\xi\smat{1&0\\0&0}\pmod{L\p^0}$
	for some $\xi\in\OO[v]$.

	Assume first that $y \in R\p^\times$, and write
	$y = \smat{a & b \\ \pip^r c & d}$.
	Letting $D = \det(y)$ we have
	\[
		y x y^{-1} =
		\frac\xi D \pmat{da & -ba \\ \pip^r dc & -\pip^r cb}
		\equiv
		\frac\xi D \pmat{da -\pip^r cb& 0 \\ 0 & 0}
		\equiv x
		\pmod{L\p^0}.
	\]
	When $y \notin R\p^\times$ we can assume
	that $y = \smat{0 & 1 \\ \pip^r & 0}$ as in \eqref{eqn:bilgen}.
	Then
	\[
		y x y^{-1}
		=
		\xi\pmat{0 & 0 \\ 0 & 1}
		= -x.
	\]

	The case $e(R\p) = -1$ follows similarly, using the identification given by
	\eqref{eqn:e=-1}.
\end{proof}

\subsubsection*{Dependence on the choices}
For completeness, we end this subsection discussing how these weight functions
depend on the choices made in their definitions.

\medskip

For weight functions of type I, let $w_v$ and $w'_v$ be two such
functions defined using $z$ and $z'$ respectively.
From the proof of Proposition~\ref{prop:wq} we know
that $w_v$ and $w'_v$ are equal up to a constant.
But $w_v(z)=1$ and $w'_v(z)=\pm1$, hence $w'_v = \pm w_v$.

\medskip

For weight functions of type II,
if $z'\in L\p$ also satisfies $\Delta(z')\in\OOx[v]$ then,
using the identifications from \eqref{eqn:e=1} and \eqref{eqn:e=-1},
$z'\equiv\xi z\pmod{L\p^0}$ for some $\xi\in\OOx[v]$.
It follows that the weight function corresponding to $z'$
equals $\chi_v(\xi)\, w_v$.

It remains to see how the choice of character
$\chi_v$ affects $w_v$. Let $\chi'_v$ be another odd character and let $w'_v$
the corresponding weight function.
Assume that $\chi_v$ and $\chi'_v$ are defined modulo $2\pip^r$ with $r \geq 1$.
Since $(1+2\pip^r\OO)^2=1+4\pip^r\OO[v]$ (see \cite[63:8]{omeara}),
there exists a function $\alpha:\OOx[v]/(1+4\pip^r\OO[v])\to S^1$
such that
\begin{equation*}
	\overline{\chi'_v}(\xi)
	=
	\alpha\left(\xi^2\,\Delta(z)\right) \,
	\overline{\chi_v}(\xi)
\end{equation*}
for all $\xi\in\OOx[v]$.
Given $x \in L\p$
there exists $\xi\in\OO[v]$ such that $x\equiv\xi z\pmod{L\p^0}$
and $\Delta(x)\equiv\xi^2\Delta(z)\pmod{4\pip^r}$.
If $\xi\notin\OOx[v]$ then $w_v(x)=w'_v(x)=0$, and if
$\xi\in\OOx[v]$ then
$w_v(x)=\overline{\chi_v}(\xi)$ and $w'_v(x)=\overline{\chi'_v}(\xi)$.
In both cases
\begin{equation*}
	w'_v(x)
	=
	\alpha(\Delta(x))\,w_v(x).
\end{equation*}
In summary, a different choice of $\chi_v$ multiplies $w_v(x)$ by a
\emph{factor of absolute value $1$ that depends only on $\Delta(x)$}.

\begin{rmk}
\label{rmk:oddchi}
The argument above works the same if $\chi'_v$, instead of an odd
character, is an arbitrary \emph{function}
$\chi'_v : \OOx[v]/ (1+2\pip\OO[v]) \to \{\pm1\}$ satisfying that
$\chi'\p(-x)=-\chi'\p(x)$ for all $x$,
thus avoiding the need to do arithmetic with roots of unity coming from the
values of $\chi\p$. 

The theta series computed with this alternative
weight function would not satisfy
Propositions~\ref{prop:modularity} and \ref{prop:hecke_linear}
when $\chi'\p$ is not a character. Nevertheless,
the Fourier coefficients of the modified theta series will be the same up to
multiplication by a complex number of absolute value $1$, hence
Theorem~\ref{thm:formula} would still be valid.
Moreover one can recover the original theta series corresponding to
a character $\chi\p$ dividing the Fourier coefficients by the values
of the function $\alpha$ as defined above.

\end{rmk}

\subsection{Adelic weight functions}

We let $w : \What \to \C$ be the (adelic) weight function given by
\begin{equation}
	\label{eqn:wadelic}
	w(x) = \begin{cases}
		\prod_{v \mid \cond l \cond\chi} w\q(x\q)
        & \text{if $x \in \widehat{L}$,} \\
		0 & \text{otherwise}.
	\end{cases}
\end{equation}
It should be denoted by $w\ssl$, but we avoid this for the sake of a lighter
notation.
Note that when $l \in \Fc$ and $\cond{\chi} = \OO$ the function $w$ is
simply the characteristic function of $\widehat{L}$.

Given $\id a \in \IF$ we let $w(\vardot;\id a)$
denote the weight function supported on $\id{a}^{-1}\widehat L$
given by
\[
	w(x;\id a)
	=
	\left(\chi\,\chic{l}\right)(\xi)\,w(\xi x),
\]
where $\xi \in \hatx F$ is such that $\xi \OO = \id a$.
Note that, by \eqref{eqn:wqhomog} and \eqref{eqn:wphomog},
if $\id m \subseteq \OO$ is prime to $2 \id N \cond l$
we have that
\begin{equation}
	w(\vardot; \id m \id a) \vert_{\id{a}^{-1}\widehat L}
	= \left(\chi_*\,\chic[*]{l}\right)(\id m)\,w(\vardot; \id a).
	\label{eqn:wsegda}
\end{equation}

\subsection{Theta series}

For each $x \in \hatx B$ let $L_x \subseteq W$ be the lattice given by
$L_x = R_x / \OO$.
Then we consider the weight function $w_x : \What \to \C$ supported on
$\widehat{L_x}$ and given by
\begin{align}\label{eqn:wx}
	w_x(y) = \chic{l} (\norm x) \, w(x y x^{-1}).
\end{align}
Then by \eqref{eqn:wqtransp} and \eqref{eqn:wptransp} we have that
\begin{equation}\label{eqn:transf_wx}
	w_{z x \gamma}(\gamma^{-1} y \gamma) = \delta\ssl(z)\,w_x(y)
	\qquad \forall\, z \in \Bil(R), \, \gamma \in B^\times.
\end{equation}

For $x \in \hatx B$ and $P \in V_\mathbf{k}$ we consider the theta series
$\tlu_{x,P}$ given by
\begin{equation}\label{eqn:theta}
	\tlu_{x,P}(z)
	= \sum_{y\in \id b^{-1} L_x} w_x(y; \id b) \, P(y)
	\, \exp_F(\disc(y) / l , z/2)\,.
\end{equation}
Here $\id b$ is the unique ideal such that $(l,\id b) $ is a
fundamental discriminant.
Note that, since $\id b^{-1} L_x$ has the involution $y\mapsto-y$,
this theta series is trivially zero unless
\begin{equation}\label{eqn:heegnerS}
	\chi_\mba(-1) \, \sgn(l) \, (-1)^\mbk = 1
	\,.
\end{equation}
Note that by
Proposition \ref{prop:modularity} below this is consistent with Remark
\ref{rmk:mholozero}.
Furthermore, using \eqref{eqn:transf_wx} we get that
\begin{equation}
	\tlu_{z x\gamma,P\cdot \gamma} = \delta\ssl(z)\,\tlu_{x,P}\, \qquad
	\forall\, z \in \Bil(R),\, \gamma\in B^\times. \label{eqn:theta_transf}
\end{equation}

\begin{rmk}
	\label{rmk:computations}
	If the classes in $\cl(R)$ are represented by elements $x \in \hatx B$
	such that $x\p \in \Bil(R\p)$ for every $\vii$ and $x\q \in
	R\q^\times$ for every $v \mid \cond l$, then the weight functions $w_x$
	are supported on $\widehat L$ for every $x$; moreover,
	by \eqref{eqn:wqtransp} and \eqref{eqn:wptransp} they are given by 
	\begin{equation}
		\label{eqn:wxw}
		w_x(y) =
		\chic[*]{l}(\norm x) \,
		\Bigl(\prod_{\vii} \iota\p(x\p)\Bigr) \,
		w(y)
		\,.
	\end{equation}
	This can be useful for making explicit computations.
\end{rmk}

\section{Modularity and the theta map}\label{sect:modularity}

The results in this section generalize \cite[Sects.\ 3 and~4]{w2}
in order to admit (type II) weight functions and skew-holomorphic forms.
We focus on the key modifications needed in this more general setting,
referring the reader to \cite{w2} for more details.

\medskip

We let $l$ and $\chi$ be as in Sect.~\ref{sect:construction}.
We consider the function $\ah \in \{\pm1\}^\mba$ given by $\ah_v = -\sgn(l_v)$.
We denote
\newcommand{\Nchi}{{\id N}_\chi}
$
	\Nchi
	=
	\lcm\left(\id N, {\cond{\chi}}^{\!\!2}\right).
$

\begin{prop}\label{prop:modularity}

	Let $\tlu_{x,P}$ be the theta series defined by \eqref{eqn:theta}. Then

	\begin{enumerate}

		\item We have $\tlu_{x,P} \in \Mhku$.
		\item For every $D$ such that $-lD \in F^+\cup\{0\}$ and for every
			$\id a \in \IF$ we have that
			\begin{equation}\label{eqn:cfourier}
				\lambda\left(D,\id a; \tlu_{x,P}\right) = \frac1{\idn a}
				\sum_{y \in \mathcal{A}_{\Delta,\id c}(L_x)}
				w_x(y;\id c) \, P(y) \,,
			\end{equation}
			where $\Delta = lD$ and $\id c = \id{ab}$, and
			\[
				 \mathcal{A}_{\Delta,\id c}(L_x) =
			 \set{y \in \id c^{-1}L_x}{\disc(y) = \Delta}.
			\]
		\item If $l\notin \Fc$, $\cond{\chi} \neq \OO$ or $\mbk \neq \mathbf{0}$,
			then $\tlu_{x,P}$ is cuspidal.
	\end{enumerate}

\end{prop}

In order to prove this modularity result we consider the left action of
a certain subgroup of $\SL_2(F)$ on the Schwartz--Bruhat space
$\mathcal S(\What)$
as in \cite[Sect.~11]{shim-hh}.
This action depends on the bilinear form $(x,y) \mapsto \ll{x,y}/l$
and is denoted by $(\beta, \eta) \mapsto {^{\beta}\eta}$.
The following results are proved in
\cite[Lem.~11.6 and Prop.~11.7]{shim-hh}.

\begin{prop}
Let $\eta \in \mathcal S(\What)$.
\begin{enumerate}
\item Let $\beta = \left(\begin{smallmatrix} 1 & b \\ 0 & 1
\end{smallmatrix} \right)$, with $b \in F$. Then
\begin{equation}\label{eqn:weil_parab}
	{^{\beta}{\eta}(x)}
	= \exp_\mbf \left(b\,\Delta(x)/2l\right) \,\eta(x).
\end{equation}
\item Let $\iota = \left( \begin{smallmatrix} 0 & -1 \\ 1 & 0\end{smallmatrix}
\right)$. Then
\begin{equation}\label{eqn:weil_fourier}
	{^{\iota}\eta}(x) = c \int_{\What} \eta(y)
	\exp_\mbf \left(-\ll{x,y}/2l\right) \, dy,
\end{equation}
where $c$ is a nonzero constant depending on the bilinear form.
\end{enumerate}

\end{prop}

We assume for simplicity that $x=1$ in \eqref{eqn:theta}, and denote
$\vartheta_P = \tlu_{1,P}$.
The function $\eta$ we consider from now on is
\[
	\eta(x) =
	\begin{cases}
        w(x;\id b) & \text{if $x \in \widehat{L}$}, \\
		0 & \text{otherwise}.
	\end{cases}
\]
Then, following the notation from \cite[Sect.~11]{shim-hh},
$\vartheta_P$ equals the theta series $z \mapsto f(z,\eta)$
defined by the function $\eta$ (and the polynomial $P$).

\begin{lemma}\label{lem:eta_aux}
	Let $x \in \What$.
	\begin{enumerate}
		\item We have $\eta(\xi x) = \left(\chi\chic[\cond l]{l}\right)(\xi^{-1}) \,
			\eta(x)$ for every $\xi \in \hatOOx$.
		\item If $\eta(x) \neq 0$, then
			$\Delta(x) / l \in \hatOO$.
		\item We have ${^{\iota}\eta}(\xi x) = \left(\chi\chic[\cond l]{l}\right)
			(\xi) \, {^{\iota}\eta}(x)$ for every $\xi \in \hatOOx$.
		\item If ${^{\iota}\eta}(x) \neq 0$, then
			$\Delta(x) / l \in
			\id{d}^{-2}\left(\Nchi\right) ^{-1} \hatOO$.
	\end{enumerate}
\end{lemma}

\begin{proof}

The claim in (a) follows from \eqref{eqn:wqhomog} and \eqref{eqn:wphomog}.
The claim in (b) follows by the definition of the type I local weight functions.
Furthermore, (c) follows from (a) and \eqref{eqn:weil_fourier}.

To prove (d) we consider the lattice $\Lambda \subseteq W$ given by
\[
	\Lambda\p =
	\begin{cases}
        L\p^0 & \text{if $\vii$,}\\
		L\p   & \text{otherwise}.
	\end{cases}
\]
Propositions \ref{prop:wq} and \ref{prop:wp} imply that $\eta$ is $l\id b
\widehat \Lambda$-periodic, so using \eqref{eqn:weil_fourier} we get that
\[
	{^{\iota}\eta}(x)
	= \exp_{\mathbf{f}}\left(-\ll{x,y}/2\right) \, {^{\iota}\eta}(x)
\]
for every $y \in \id b \widehat\Lambda$.
Hence if ${^{\iota}\eta}(x) \neq 0$ we have that
$\ll{x,y}/2 \in \id{d}^{-1}$ for every $y \in \id b \widehat\Lambda$,
i.e., $x/2 \in (\id d \id b \widehat\Lambda)^\sharp = (\id d \id b)^{-1}
\widehat\Lambda^\sharp$.
Let $\zeta \in \hatx F$ be such that $\zeta \OO = \id d \id b$.
We claim that
\begin{equation}\label{eqn:Ddual}
	\Delta(\zeta x) \in \left(\Nchi\right)^{-1} \hatOO\,,
\end{equation}
which we verify locally.
Assume first that $e(R\p) = 1$.
Under the identification from \eqref{eqn:e=1} we see that if $z \in
\Lambda\p^\sharp$ then
\[
	z =
	\pmat{a/4\pip & b/4\pip^r \\ c/4 & 0}
\]
with $a,b,c \in \OO[v]$.
Then $\Delta(2z) = a^2/4\pip^2 + bc/\pip^r$, which implies that $\Delta(2z)
\in \left(\Nchi\right)^{-1} \OO[v]$, thus giving \eqref{eqn:Ddual}.
If $e(R\p) = -1$ a similar reasoning applies using \eqref{eqn:e=-1}.

Finally, since $\cond l$ is prime to $\id{N}$ \eqref{eqn:Ddual} implies that
$\Delta(\zeta x)$ is $\cond l$-integral.
To prove~(d) it remains to show that
$\cond l\mid\Delta(\zeta\q x\q)$.
Since $\cond l$ is square-free,
it suffices to show that $v
\mid \Delta(\zeta\q x \q)$ for every $v \mid \cond l$,
which follows as in \cite[Lem.~3.1]{w2} with no modifications.
\end{proof}

\begin{lemma}\label{lem:cusp_eta}

	Assume that $l \notin \Fc$ or that $\cond{\chi} \neq \OO$.
	Let $\beta \in \SL_2(F)$.
	Then ${^\beta}\eta(0) = 0$.

\end{lemma}

\begin{proof}
	By \cite[Lem.~3.3]{w2} it suffices to prove that ${^{\beta}}\eta\p (0) =
	0$ for every $\beta \in \SL_2(F)$ and for $\vii$.
	Let $\beta = \smat{a & b \\ c & d}$ and write 
	\[
		\beta
		=
		\pmat{ 1/c & a \\ 0 & c }
		\iota
        \pmat{1& d/c \\ 0 & 1}
		.
	\]
	By \cite[Prop.~11.5]{shim-hh} and \eqref{eqn:weil_parab} we can assume that
	$c=1$ and $a = 0$.
	Using \eqref{eqn:weil_parab} again and \eqref{eqn:weil_fourier} we get that
	\[
		{^\beta}\eta\p(0)
		=
		c \int_{W\p} \exp_v \left(d\Delta(x)/2l\right) \eta\p(x) \,dx.
	\]
	Then the result follows, since item (a) of Lemma~\ref{lem:eta_aux}
	implies that $\eta\p$ is an odd function if $\vii$.
\end{proof}

\begin{prop}\label{prop:weil_gral}

Let $\beta \in \Ga{4\Nchi}$.
Then
$
	{^{\beta}\eta} =
	\overline{\chi_{\cond{\chi}}}(a_\beta) \,
	\chic[\cond{-1}]{-1}(a_\beta) \,
	\eta .
$
\end{prop}

\begin{proof}
Following the proof of \cite[Prop.~3.4]{w2}, where we take $u=-1$, we get that
there exists $\xi \in \hatOOx$ with $\xi_v = a_\beta$ for every $v \mid
2\id N\cond{-l}$ such that
\begin{equation*}
	{^{\beta}\eta}(x) = 
	\chic[\cond{-l}]{-l}(a_\beta) \,
	\eta(\xi x)
\end{equation*}
for every $x \in \widehat{W}$.
Since $\cond l$ is prime to $\cond{-1}$ we have that $\cond{-l} = \cond l
\cond{-1}$.
Furthermore, we have that $\chic[\cond l]{-1}(a_\beta) = 1$ and $\chic[\cond{
-1}]{l}(a_\beta) = 1$, since $a_\beta \in \OOx[\pv]$ for $\pv\mid \cond{-l}$.
These facts imply that $\chic[\cond{-l}]{-l} (a_\beta) \chic[\cond{l}]{l} (a_\beta)
= \chic[\cond{-1}]{-1} (a_\beta)$.
Hence the result follows using part (a) of Lemma~\ref{lem:eta_aux},
which implies that
\[
	\eta(\xi x)
	=
	\left(\chi_{\cond{\chi}}\chic[\cond l]{l}\right)(a_\beta^{-1}) \, \eta(x).
\]
\end{proof}

\begin{proof}[Proof of Proposition~\ref{prop:modularity}]

The modularity of $\vartheta_P(z) = f(z,\eta)$ follows,
as in the proof of \cite[Prop.~3.1~(a)]{w2},
using Proposition \ref{prop:weil_gral} and \cite[Prop.~11.8]{shim-hh},
which asserts that for every $\beta \in \Ga{4 \Nchi}$ we have 
\begin{equation*}
	f\left(\beta z,{^\beta}\eta\right)
	=
	\chic[\cond{-1}]{-1}(a_\beta) \,
	K_\ah(\beta,z)\, f(z),
\end{equation*}
since the automorphy factor defined in \cite[(11.19)]{shim-hh} equals
$\chic[\cond{-1}]{-1}(a_\beta) \, K_\ah(\beta,z)$.

The proof of \eqref{eqn:cfourier} follows as in \cite[Prop.~3.1~(b)]{w2} with no
modifications and
the cuspidality of $\vartheta_P$
follows from Lemma~\ref{lem:cusp_eta}
as in the proof of \cite[Prop.~3.1~(c)]{w2}.
\end{proof}

We consider the theta map $\tlu : \MkR \to \Mhku$ given by
\begin{equation}\label{eqn:theta_map}
	\tlu(\varphi) =
	\sum_{x\in \cl(R)} \tfrac{1}{t_x}\,\tlu_{x,\varphi(x)}.
\end{equation}
This map is well defined by \eqref{eqn:phi=sum_phix} and
\eqref{eqn:theta_transf} and it satisfies $\tlu(\varphi_{x,P}) =
\tlu_{x,P}$ for every $x \in \hatx B$ and $P \in V_\mathbf{k}$.

We consider the action of the Hecke operators on $\MkR$ as defined in
\cite[Sect.~2]{w2} and we normalize the standard Hecke operator $T\pp$ on
$\Mhku$ multiplying it by $\idn p$ as in \cite[Sect.~4]{w2}.
The following result will be used in Sect.~\ref{sect:shimura}.

\begin{prop}\label{prop:hecke_linear}

	The map $\tlu$ satisfies
	\begin{equation}\label{eqn:hecke}
		T\pp\left(\tlu(\varphi)\right)
		=
		\chi_*(\id p) \,
		\tlu\left(T\pp(\varphi)\right)
	\end{equation}
	for every $\id p \nmid 2 \id N \cond l$.
	Furthermore, it maps cuspidal forms to cuspidal forms.

\end{prop}

\begin{proof}
	The first claim follows using \eqref{eqn:quat_hecke} and the formulas for the
	action of the Hecke operators on Fourier coefficients given in
	\cite[Prop.~5.5]{shim-hh} and \eqref{eqn:cfourier}, as in
	\cite[Prop.~4.1]{w2}.
	The only modification required in that proof is that now, by
	\eqref{eqn:wqtransp}, \eqref{eqn:wsegda} and~\eqref{eqn:wx},
	the weight functions satisfy
	\[
		w_x(y;\id{pc})
		=
		\chi_*(\id p) \,
		w_{hx}(y;\id c)
		\qquad
		\forall\, h \in H_{\id p},\, y \in \id c^{-1}L_{hx}
	\]
	for every $\id p \nmid \id N$,
	which explains the $\chi_*(\id p)$ factor in \eqref{eqn:hecke}.

	The second claim follows combining Proposition~\ref{prop:modularity} (c)
	and \cite[Prop.~2.5]{waldspu}.
\end{proof}

\newcommand{\ttemp}{f}

For the remainder of this section we denote $f = \tlu(\varphi)$.

\begin{prop}\label{prop:kohnen}
	Let $D \neq 0$ be such that
	$\lambda\left(D,\id a;\ttemp\right) \neq 0$.
	Then $(D,\id a)$ is a discriminant.
    Furthermore, for $v \mid \id N$:
\begin{enumerate}
	\item If $v \nmid D\id a^2 $ then
		$\kro Dv = \kro{l}{\pv} e(R\p)$.
	\item If $v \mid D\id a^2$ then $v\nmid\cond\chi$.
	\item If $v \mid D\id a^2$ and $(D,\id a)$ is fundamental 
		then $\val_v(\id N) = 1$.
\end{enumerate}

\end{prop}

\begin{proof}

By \eqref{eqn:theta_map} we can assume that $f = \vartheta\ssl_{x,P}$
and without loss of generality $x=1$.
By \eqref{eqn:cfourier} if $\lambda(D,\id a;f) \neq0$ there exists 
$y \in \mathcal{A}_{\Delta, \id c}(L)$ such that $w(y;\id c) \neq 0$.
The former implies that $(\Delta, \id c)$ is a discriminant and the
latter that $\Delta(y)\id c^2\subseteq\cond l=l\id b^2$.
Since $\Delta(y)=lD$ and $\id c=\id a\id b$
we have that $D\in\id a^{-2}$ and
then \cite[Prop.~3.6]{w2} implies that $(D, \id a)$ is a discriminant.

Let $v \mid \id N$, so that in particular $\kro lv \neq 0$.
Let $\xi\p$ and $\zeta\p$ be local generators for $\id a$ and $\id b$ respectively.
Then $z_v=\xi\p\,\zeta\p\, y\in R\p$ and
\[
	\Delta(z_v) = \left(D\,\xi\p^2\right)\left(l\,\zeta\p^2\right).
\]
In particular $v \nmid D\id a^2$ if and only if $\Delta(z_v) \in \OOx[v]$.
Hence (a) follows from Proposition~\ref{prop:disc_eichler}
and (b) follows from \eqref{eqn:sop_wii}.
When $(D,\id a)$ is fundamental $\OO[v][z_v]$ is a
maximal order.
If $\val_v(\id N) > 1$ then
Proposition~\ref{prop:disc_eichler}
implies $v\nmid D\id{a}^2$, proving (c).
\end{proof}

\begin{rmk}\label{rmk:kohnen}

Assume that $\chi = 1$ and $f$ is cuspidal.

In the case $F = \Q$, let $u = -\sgn(l)$ and let $\ttemp \cdot u$ be the
holomorphic form given by Proposition \ref{prop:hilbert_holomorphic}.
When $\id N = (N)$ is odd and square-free, Proposition \ref{prop:kohnen} (a)
implies that $\ttemp \cdot u \in S_{k+3/2}^{\pm,p}(N,\chi^u)$ for every $p \mid N$,
the subspace of the Kohnen space defined in
\cite[Prop.~4]{kohnen-newforms} with $\pm = \kro{l}{p} e(R_p)$.
For general $N$ we have that $\ttemp \cdot u \in S^{\emptyset,\kappa}$,
the subspace of the Kohnen space considered in \cite{ueda}
corresponding to the function
$\kappa:\set{p \mid N}{p \neq 2,\, p^2 \mid N} \to \{\pm 1\}$
given by $\kappa(p) = \kro{l}{p} e(R_p)$.

For arbitrary $F$, assuming that there exists $u \in \OOx$ with $lu$
totally negative, we consider the holomorphic form $\ttemp \cdot u$ given by
Proposition \ref{prop:hilbert_holomorphic}.
When $\id N$ is odd and square-free, by \eqref{eqn:fu} and
Proposition~\ref{prop:kohnen} (a),
we have that $\ttemp\cdot u$ belongs to the Kohnen space as
defined in \cite{hiraga-ikeda} and \cite{su}.

\end{rmk}

\section{The main formula}\label{sect:formula}

We devote this section to prove Theorem \ref{thm:formula}.
Its proof is similar to that of \cite[Thm.~7.1]{w2}.
Here we introduce the needed tools and results, stressing the modifications
needed in this new setting.

\medskip

As in the Introduction we let $g$ be a normalized cuspidal newform of level
$\id N$ and weight $\mathbf{2+2k}$
and we let $\gamma : \SN \to \{\pm1\}$.
We let $R$ be as in Sect.~\ref{sect:construction},
and let $l$ and $\chi$ satisfying
the hypotheses mentioned in the Introduction, namely
\begin{enumerate}[label=\textbf{H$l$}.,ref={\textbf{H$l$}},
                  labelindent=\parindent,leftmargin=*,
                  topsep=\medskipamount
                  ]
	\item $\kro lv = \gamma(v)\,e(R_v)$ for all $v\in\SN$,
		and $\kro lv \neq 0$ for all $v \mid 2$,
\end{enumerate}
\begin{enumerate}[label=\textbf{H$\chi$}.,ref={\textbf{H$\chi$}},
				   labelindent=\parindent,leftmargin=*,
				   topsep=\medskipamount
				   ]
		   \item For finite $v$ the character $\chi_v$ is unramified if $v
			   \nmid \id N$ or $\gamma(v)^{\val_v(\id N)} = \AL[v]{g}$,
		 and is odd otherwise.
\end{enumerate}
Note that this includes the assumptions made in Sects.~\ref{sect:construction}
and \ref{sect:modularity}.

\begin{rmk}
\label{rmk:gamma_rest}
In \cite{w2} we considered only the case $\chi = 1$, which restricted the
construction to types $\gamma$ satisfying
\begin{equation*}
	\gamma(v)^{\val_v(\id N)} = \AL[v]{g}, \quad v \mid \id N.
\end{equation*}
For instance, for squarefree level, only one type $\gamma$ satisfies this
for a given $g$.
\end{rmk}

Under the hypotheses above $\tlu(\varphi)$ satisfies the following
facts in terms of~$\gamma$.
First, $\tlu(\varphi)$ is holomorphic in the variables $v
\in \mba$ such that $\gamma(v) = 1$ and antiholomorphic at the remaining
variables.
Furthermore, \eqref{eqn:heegnerS} is equivalent to $\AL{g,\gamma}=1$.
Finally, Proposition \ref{prop:kohnen} implies that if $\lambda(D,\id
a;\tlu(\varphi)) \neq 0$ and $(D,\id a)$ is a fundamental discriminant 
then $D$ is of type $\gamma$.

\medskip

Let $c(\mathbf{k})$ and $C(\id N)$ denote the positive rational numbers given by
\[
	c(\mathbf{k}) =
	\prod_{v\in\mathbf{a}} \frac{r_{k_v}}{s_{k_v}},
	\qquad
	C(\id N) =
	\prod_{v \mid \id N} (\mathcal N(v) + 1) \,\mathcal N(v)^{\val\p(\id N) - 1}\,,
\]
where for $k \in \Z_{\geq 0}$ we denote
\begin{equation*}
    r_k = \frac{2^{2k+1}(k!)^2}{(2k)!}\,,
	\qquad
	s_k = \frac{1}{\Gamma(k+1/2)}\,
		\sum_{q=0}^{\left\lfloor \frac{k}{2} \right\rfloor} 
		\:\frac{\Gamma(k+1/2-q)}{q!\,(k-2q)!\,2^{2q}}\,.
\end{equation*}
For later use we define
\begin{equation} \label{eqn:ckn}
	\ckN
	=
	\frac{\dF^{\;1/2}}{\hF}\,
	\frac{c(\mathbf{k})\:C(\id N)}{2^{2\omegaN}}
	>0
	\,,
\end{equation}
where $\omega(\id N)$ is the number of prime divisors of $\id N$.

\medskip

Let $D \in F^\times$ of type $\gamma$.
Denote $\Delta = lD$ and let $K = F(\sqrt\Delta)$.
Denote by~$\Xi$ the character of $C_K$ corresponding to the quadratic
extension $F(\sqrt{l},\sqrt{D}) / K$ and
by~$P_\Delta$ the polynomial defined by \cite[(4.3)]{w2}.
Consider the quaternionic form
\begin{equation*}
    \psi^l_D
	=
	\tfrac1{\tK}
	\sum_{a\in\cl(K)} \Xi(a) \varphi_{a,P_\Delta}\qquad
	\in \mathcal{M}_{\mathbf{k}}(R)\,.
\end{equation*}

We now recall \cite[Thm.~5.1]{w2}, which relates
the central value $L_{l,D}(1/2,g)$ to the height
of the $g$-isotypical component of~$\psilD$.
It is based on \cite[Thm.~1.3.2]{zhang-gl2} (for $\mbk = \mathbf0$) and on
\cite[Thm.~1.2]{Xue-Rankin} (for general $\mbk$, which requires $\id N \neq \OO$).

\begin{thm}[Theorem 5.1 of \cite{w2}]
	\label{thm:xue}
Let $D \in F^\times$ of type $\gamma$.
Let $\Tg$ be a polynomial in the Hecke operators prime to $\id N$ giving the 
$g$-isotypical projection.
If $\mbk \neq \mathbf0$, assume that $\id N\neq\OO$.
Assume that $\cond{\Delta}$ is prime to $2 \id N$ and
that $\cond l$ and $\cond D$ are prime to each other.
Then
\begin{equation*}
 L_{l,D}(1/2,g) =
 \ll{g,g}
 \:
 \frac{\dF^{\;1/2}}{\hF}
 \:
 \frac{c(\mathbf{k}) \: C(\id N)}
 {\mathcal N(\cond \Delta)^{1/2}(-\Delta)^\mathbf{k}}
 \:
 \ll{\Tg\psi^l_D, \psi^l_D}\,.
\end{equation*}

\end{thm}

\begin{rmk}\label{rmk:capitulation}
	The right hand side of the formula in \cite[Thm.~5.1]{w2}
	contains a wrong factor $m_K^2$.
	This is caused
	by a mistake in the proof of
	\cite[Prop.~3.5]{waldspu}, where we claimed that 
	$\vert K^\times/F^\times \cap \hatx R/\hatOOx\vert = \tK$.
	Actually, since $\OO[K] \subseteq R$, we have that
	\[
		\vert K^\times/F^\times \cap \hatx R/\hatOOx\vert
		= \vert K^\times/F^\times \cap \hatOOx[K]/\hatOOx\vert
		= m_K \, \tK.
	\]
	The last equality is explained by \cite[p.~548]{rvyang} where it is shown
	that there exists an exact sequence
	\[
		1
		\to
		\OOx
		\to
		\OOx[K]
		\to
		K^\times/F^\times \cap \hatOOx[K]/\hatOOx
		\to
		\hatx F / \OOx F^\times
		\to
		\hatx K / \OOx[K] K^\times
		.
	\]
\end{rmk}

\bigskip

In order to relate central values with Fourier coefficients
we consider the set $X_{\Delta,\id c}$ of special points
associated to the discriminant $(\Delta,\id c)$ (see \cite[Sect.~4]{w2}) and we
let
\begin{equation*}
	\etalD
	=
	\tfrac1{w(\omega;\id c) \, t_K}
	\sum_{x \in X_{\Delta, \id{c}}}
		w_x(\omega;\id c) \,
		\varphi_{x,P_{\Delta}}
	\quad \in \mathcal{M}_{\mathbf{k}}(R)\,,
\end{equation*}
where $\omega \in K$ is such that $\OO \oplus \omega \id c$ is an order in $K$
and $\Delta(\omega) = \Delta$.

\begin{prop}\label{prop:coef_serie_theta}
Let $\varphi\in \mathcal M_\mathbf{k}(R)$.
Let $D \in F^\times$ be such that $-lD\in F^+$ and let $\id a \in \IF$. 
Then
\begin{equation*}
 \lambda\left(D,\id a;\tlu(\varphi)\right) =
 \frac1{\idn a}\,\ll{\varphi, \etalD}\,.
\end{equation*}
\end{prop}

\begin{proof}
This is \cite[Prop.~4.2]{w2}, and holds with no changes.
\end{proof}

From here on we assume that $(D, \id a)$ is fundamental.
Then $\etalD$ is related to $\psilD$ by the following result.

\begin{prop}\label{prop:eta=proy_psi}
 
Let $\delta\ssl$ be as in \eqref{eqn:deltal}.
Assume that $\cond{\Delta}$ is prime to $\id N$. Then
\begin{equation*}
	\etalD
	=
	\sum_{z \in \Bil(R)} \delta\ssl(z) \, \psilD \cdot z\,.
\end{equation*}
In particular,
$\etalD \in \mathcal{M}_\mathbf{k}(R,\mathbbm{1})^{\delta\ssl}$.

\end{prop}

\begin{proof}
This follows using the same proof as that of \cite[Prop.~6.3]{w2}.
The key fact in that proof is given here by~\eqref{eqn:transf_wx}.
\end{proof}

Let $\pi = \otimes_v \pi_v$ be the irreducible automorphic representation of
$\PGL_2(\FA)$ corresponding to $g$.
For every $\pv \in \mbf$ where $B$ is ramified
we have that $\val\p(\id N)$ is odd,
hence $\pi_v$ does not belong to the principal series.
It follows that there exists an irreducible automorphic representation
$\pi_B$ of $\hatx B$
corresponding to $\pi$ under the Jacquet-Langlands map.
In \cite[Prop.~8.6]{gross-local} it is shown that
$\hatx R$ fixes a unique line in the representation space of $\pi_B$.
This line gives an explicit quaternionic modular form
$\varphi_g \in \mathcal{S}_\mathbf{k}(R,\bbu)$,
which is well defined up to a constant.

\begin{lemma}\label{lem:sect5}
Let $\varphi_g$ be as above.
Then
$\varphi_g\in \mathcal{S}_\mathbf{k}(R,\mathbbm{1})^{\delta\ssl}$.
\end{lemma}

\begin{proof}
By \ref{hyp:Hl} and \ref{hyp:Hchi} the ramification of $B$ is given by the set
\begin{equation}\label{eqn:Sigma}
	\Sigma
	=
	\mba \cup
	\set{\pv \mid \id N}
	{\chi_v(-1) \tkro lv^{\!\val\p(\id N)} = -\varepsilon_g(v)}.
\end{equation}

Let $\pv \mid \id N$, and let $z\p \in N(R\p)$
be the generator for $\Bil(R\p)$ given in \eqref{eqn:bilgen}.
Since $\norm(z\p) = - \pip^{\val\p(\id N)}$ and $\pv \nmid \cond l$,
\[
	\delta\ssl(z\p) =
		\chi_v(-1) \tkro lv^{\!\val\p(\id N)}.
\]
Furthermore, by \cite[Thm.~2.2.1]{roberts-thesis} $z\p$ acts on $\varphi_g$ by
\[
	\varphi_g \cdot z\p =
	\begin{cases}
        -\AL[\pv]{g} \varphi_g & \text{if $\pv \in \Sigma$,} \\
    \AL[\pv]{g} \varphi_g & \text{if $\pv \notin \Sigma$.}
	\end{cases}
\]
These facts together with \eqref{eqn:Sigma}
complete the proof.
\end{proof}

We are now ready to prove Theorems \ref{thm:main} and~\ref{thm:formula}.
Both proofs require the following result of Waldspurger,
which we state for convenience of the reader.

\begin{thm}[{\cite[Thm. 4]{waldspu_forum}}]
	\label{thm:wforum}
Let $h$ be a cuspidal newform such that the sign of the functional equation of
$L(s,h)$ equals $1$.
Let $\Sigma$ be a finite set of places of $F$,
and for every $v \in \Sigma$ let $\delta_v > 0$.
Then there exists $\xi \in F^\times$ such that
\begin{enumerate}
	\item $\abs{\xi-1}_v < \delta_v$ for every $v \in \Sigma$.
	\item $L\left(h\otimes\chic{\xi},1/2\right) \neq 0$.
\end{enumerate}
\end{thm}

\begin{proof}[Proof of Theorem~\ref{thm:formula}]
We first assume that $\cond D$ is prime to $2 \cond l \id N$.
Combining Lemma \ref{lem:sect5} with Proposition \ref{prop:eta=proy_psi} we have
\begin{equation}\label{eqn:phietapsi}
   \ll{\varphi_g,\etalD}
   = 2^{\omegaN}
   \ll{\varphi_g,\psilD}
   \,.
\end{equation}
Since $\Tg\psilD$ is the $\varphi_g$-isotypical projection of $\psilD$ we have that
$\Tg\psilD = 
\frac{\ll{\psilD,\varphi_g}}{\ll{\varphi_g,\varphi_g}} \, \varphi_g$.
Then using Proposition \ref{prop:coef_serie_theta} and \eqref{eqn:phietapsi} we
get
\[
	\ll{\Tg\psilD, \psilD} = 
    \frac{\abs{ \ll{\psilD,\varphi_g} }^2}
         {\ll{\varphi_g,\varphi_g}}
    =
    \frac{\abs{ \ll{\etalD,\varphi_g}}^2}
         {2^{2\omegaN} \,
          \ll{\varphi_g,\varphi_g}}
    =
    \frac{\idn a^2}{2^{2\omegaN}} \,
	\frac{\abs{\lambda\left(D,\id a; \vartheta\ssl(\varphi_g)\right)}^2}
         {\ll{\varphi_g,\varphi_g}}\,.
\]
The result follows, under the assumption on $\cond D$, from Theorem \ref{thm:xue}.
\par\medskip
Now we remove the assumption on $\cond D$.
Given $D\in F^\times$ of type $\gamma$ pick $D' \in F^\times$ of type
$\gamma$ such that $\cond{D'}$ is prime to $2\cond{l}\id N$ and such that
\[
	L\left(g\otimes\chic{D'},1/2\right) \neq 0,
\]
which is possible by Theorem~\ref{thm:wforum}.
We have already proven the formula for $D'$;
hence Proposition~\ref{prop:bmw} below implies
it also holds for $D$.
\end{proof}

\begin{rmk}
Theorem~\ref{thm:formula} generalizes \cite[Thm.~7.1]{w2} to arbitrary $\gamma$.
We must point out that there is a mistake in the formula for
$c_\Delta$ given there.
The correct formula is $c_\Delta=\idn{a/b}$
(see also Remark~\ref{rmk:capitulation}).
\end{rmk}

\begin{proof}[Proof of Theorem~\ref{thm:main}]
Since $\varepsilon_{g,\gamma} = 1$,
assumption \ref{hyp:Hl} implies that the sign of the
functional equation for $L\left(s,g\otimes\chi^l\right)$ equals $1$.
Hence by Theorem~\ref{thm:wforum}
there exists $l$ satisfying that assumption and such that
$L\left(1/2,g\otimes\chi^l\right) \neq 0$.
Choose $\chi$ satisfying \ref{hyp:Hchi} of the smallest possible
conductor, i.e.,
\(
	\cond{\chi} = 
	\prod_{\id p \in S} \id p^{e_v+1}
\),
where $S = \set{v \mid \id N}{\gamma(v)^{\val_v(\id N)} \neq\varepsilon_g(v)}$
and $e_v = \val_v(2)$.
With these choices Theorem~\ref{thm:main} follows
from Theorem~\ref{thm:formula},
taking $f\ssg = \vartheta\ssl(\varphi_g)\neq 0$ and
\begin{equation}
	\label{eqn:c_gamma}
	c_{g,\gamma} =
	\frac\ckN{L\left(1/2,g\otimes\chi^l\right)
	N(\cond l)^{1/2}\abs{l}^{\mathbf{k}}}
	\,
	\frac{\ll{\tlu(\varphi_g),\tlu(\varphi_g)}}{\ll{\varphi_g,\varphi_g}}
	> 0
	.
\end{equation}
The claim about the level follows using
Proposition~\ref{prop:modularity}~(a).
Finally, the cuspidality of $f\sg$ follows from
Proposition~\ref{prop:hecke_linear}.
\end{proof}

\section{The Shimura correspondence}\label{sect:shimura}

For the rest of the article we let $\psi$ be the character on $\FA$
given by $\psi(\xi) = \exp_{\mathbf A}(\xi/2)$.
For $D \in F$ we denote $\psi^D(x) = \psi(Dx)$, and we use a similar notation
locally.
Denote by $S_\psi$ the Shimura correspondence, relative to $\psi$, between near
equivalence classes of cuspidal irreducible automorphic representations of
$\widetilde{\SL_2}(\FA)$ and cuspidal irreducible automorphic representations of
$\PGL_2(\FA)$.
Furthermore, denote by $\Theta$ the theta correspondence between these
representations, as well as between their local counterparts.
Many of the notions and results we need about these correspondences are
summarized in \cite[Sect.~3]{baruch-mao}.

Let $f = \vartheta\ssl(\varphi_g)$ be the modular form of
half-integral weight appearing in Theorem~
\ref{thm:formula}, which we assume to be nonzero.
Let $\tphi = \otimes_v \tphi_v$ denote the cuspidal automorphic form
corresponding to $f$
(which we denoted by $f_\mbA$ in Sect.~\ref{sect:hilbert_half_integral}),
and denote by $\tpi = \otimes_v \tpi_v$ the cuspidal
automorphic representation of $\widetilde{\SL_2}(\FA)$ generated by $\tphi$.
Denote by $V_\tpi$ and by $V_{\tpi_v}$ the spaces of $\tpi$ and $\tpi_v$,
respectively.
As in Sect.~\ref{sect:formula} we let $\pi = \otimes_v \pi_v$ be the
irreducible automorphic representation of $\PGL_2(\FA)$ corresponding
to $g$.

\begin{prop}\label{prop:shimura_map}
	The representation $\tpi$ is irreducible and $S_\psi(\tpi) = \pi$.
\end{prop}

\begin{proof}

	The space of $\tpi$ decomposes as a direct sum of irreducible
	subspaces $T_i$, each one appearing with multiplicity one.
	Since Proposition \ref{prop:hecke_linear} implies that $f$ is a Hecke
	eigenform outside a finite set of places, the subspaces $T_i$ are, in
	particular, locally isomorphic almost everywhere.
	Moreover, these subrepresentations have the same central character, the one
	of $f$.
	Then \cite[Thm.~3]{waldspu_forum} shows that $T_i \simeq T_j$,
	which implies the first claim.

	Proposition \ref{prop:hecke_linear} implies that $\tphi$ is a
	nonzero vector in $V_{\tpi}$ with the same Hecke eigenvalues as $g
	\otimes \chi$ outside $2 \id N \cond l$.
	By \cite[Prop.~4]{waldspu-coeffs}
	this implies that
	\[
		S_\psi(\tpi)
		\simeq \left(\pi \otimes \chi\right) \otimes \chi^{-1}
		\simeq \pi,
	\]
	which proves the second claim.
\end{proof}

Given a function $\epsilon: \mba \cup
\set{v \mid \id N}{\text{$\pi_v$ is special or supercuspidal}} \to
\{\pm1\}$ we denote,
following \cite[Sect.~3]{baruch-mao},
\[
	F^\epsilon(\pi) =
	\set{D \in F^\times}{\kro{D}{\pi_v} = \epsilon(v) \text{ for every } v};
\]
the symbols $\kro{D}{\pi_v}$ are defined in \cite[p.~346]{baruch-mao}.

Here we consider the function $\epsilon$ given by
$\epsilon(v) = \gamma(v)^{\val_v(\id N)}$,
where as in the Introduction we let 
$\val_v(\id N) = 1$ for $v \in \mba$.
Note that $F^\epsilon(\pi)$ contains every $D \in F^\times$~of type $\gamma$.

\begin{prop}\label{prop:shimura_map2}
	Let $\tpie$ be the representation of $\widetilde \SL_2(\FA)$ determined by
	$\pi$ and $\epsilon$ in \cite[Thm.~3.2]{baruch-mao}.
	Then $\tpi \simeq \tpie$.

	Moreover, for every $D \in F^\epsilon(\pi)$ and for every $v$ we have that
	$\tpi_v$ has a nontrivial $\psi_v^D$-Whittaker model and
	\begin{equation}\label{eqn:thetaD}
		\Theta\left(\pi_v\otimes\chic[v]{D},\psi_v^D\right) \simeq \tpi_v.
	\end{equation}
\end{prop}

\begin{proof}
	Let $D \in F^\times$ of type $\gamma$ such that $\cond D$ is prime to $2
	\cond l \id N$ and $L(\pi\otimes\chic{D},1/2) \neq 0$
	(such $D$ exists by Theorem~\ref{thm:wforum}).
	Since $D \in F^\epsilon(\pi)$, we have that
	\begin{equation}\label{eqn:thetaD_aux}
		\Theta\left(\pi \otimes \chic{D},\psi^D\right)
		\simeq
		\tpie.
	\end{equation}
	Theorem~\ref{thm:formula}, already proved for this particular $D$,
    implies $\lambda(D,\id a; f) \neq 0$; hence
    $\tpi$ admits a
	nontrivial $\psi^D$-Whittaker model and so
    $\Theta(\tpi,\psi^D) \neq 0$.
	Since $S_\psi(\tpi) = \pi$ we have that
	$\Theta\left(\tpi,\psi^D\right) \otimes \chic{D} \simeq \pi$.
	In particular,
	$\Theta\left(\tpi_v,\psi_v^D\right) \otimes \chic[v]{D} \simeq \pi_v$
	for every $v$.
	Since the theta correspondence is locally a bijection
	(see \cite[Thm.~1]{waldspu_forum}) this proves
	\eqref{eqn:thetaD}, which together with \eqref{eqn:thetaD_aux}
    imply $\tpi \simeq \tpie$.

	Finally by \cite[Lem.~6]{waldspu_forum} we have that
	\begin{equation*}
		\Theta\left(\pi_v\otimes\chic[v]{D'},\psi_v^{D'}\right)
		\simeq
		\Theta\left(\pi_v\otimes\chic[v]{D},\psi_v^D\right)
		\simeq
		\tpi_v
	\end{equation*}
	for every $D' \in F^\epsilon(\pi)$.
	This shows
	that $\tpi_v$ has a nontrivial $\psi_v^{D'}$-Whittaker model,
	completing the proof.
\end{proof}

We use the notation from \cite{waldspu_forum} to describe local
irreducible representations.
Let $\mu$ denote a character of $F_v^\times$.
When $\mu^2 \neq \abs{\,\cdot\,}_v^{\pm1}$ 
we denote by $\pi_v(\mu,\mu^{-1})$ and by $\tpi_v(\mu)$ the
representations, of $\PGL_2(F_v)$ and $\widetilde{\SL_2}(F_v)$ respectively,
belonging to the principal series.
When $\mu^2 = \abs{\,\cdot\,}_v^{\pm1}$ we denote by
$\sigma_v(\mu,\mu^{-1})$ and by $\widetilde \sigma_v(\mu)$ the
corresponding special representations.
In the metaplectic case these representations depend on the additive character
$\psi_v$.

\begin{prop}\label{prop:tpiv}
	Let $v \in \mbf$.
	\begin{enumerate}
		\item
			Assume that $\val_v(\id N) = 0$ and write
			$\pi_v \simeq \pi_v(\mu,\mu^{-1})$.
			Then $\tpi_v \simeq \tpi_v(\mu)$.
		\item
			Assume that $\val_v(\id N) = 1$ and write
			$\pi_v \simeq \sigma_v(\mu,\mu^{-1})$.
			Then $\tpi_v \simeq \widetilde\sigma_v(\mu)$ if $\gamma(v) =
			\AL[v]{g}$, whereas $\tpi_v$ is supercuspidal if $\gamma(v) =
			-\AL[v]{g}$.
		\end{enumerate}
\end{prop}

\begin{proof}
	\leavevmode
	\begin{enumerate}
	\item
		By Proposition \ref{prop:shimura_map} and \cite[Lem.~38]{waldspu_forum}
		we have that $\tpi_v \simeq \Theta_v(\pi_v,\psi_v)$.
		Then the result follows from \cite[Prop.~4]{waldspu_forum},
		where it is proved that $\Theta_v(\pi_v,\psi_v) \simeq \tpi_v(\mu)$.
	\item
		Write $\mu = \abs{\,\cdot\,}^{1/2}_v \chi^{\tau}_v$, with $\tau
		\in \OOx[v]$.
		Pick $D \in F^\times$ such that $\kro Dv=\gamma(v)$.
		By \eqref{eqn:thetaD} we have that
		$\tpi_v \simeq \Theta(\sigma_v(\nu,\nu^{-1}),\psi_v^D)$,
		with $\nu = \mu \chi_v^D$.
		By \cite[Prop.~4]{waldspu_forum} we have that
		$\tpi_v \simeq \widetilde\sigma(\nu,\psi_v^D) \simeq
		\widetilde\sigma(\mu,\psi_v)$ if
		$\kro{\tau D}v = -1$, and $\tpi_v$ is supercuspidal otherwise.
		This proves the result, since $\kro{\tau}{v}=-\AL[v]{g}$.
\end{enumerate}
\end{proof}

\section{Extending the main formula}\label{sect:extending}

The first part of the proof of Theorem~\ref{thm:formula} given in 
Sect.~\ref{sect:formula} imposed restrictions on the discriminants considered.
As mentioned in the second part of that proof, these restrictions can be removed 
by using Theorem~\ref{thm:wforum} together with the following result,
which show that given two discriminants of type $\gamma$, then the
assertion of Theorem~\ref{thm:formula} holds for one of them if and only if it
holds for the other.

\begin{prop}\label{prop:bmw}

	Let $f=\tlu(\varphi_g)$ as in Sect.~\ref{sect:shimura}. Then
\begin{multline}
	\label{eqn:bmw}
	2^{\omega(D_1,\id N)} \,
	\idn[1]{a} \,
	\abs{ \lambda\left(D_1, \id a_1;f\right) }^2 \cdot
	L\left(g\otimes\chic{D_2},1/2\right) \abs{D_2}^{\mbk+\mathbf{1/2}} \\
	=
	2^{\omega(D_2,\id N)} \,
	\idn[2]{a} \,
	\abs{ \lambda\left(D_2, \id a_2;f\right) }^2 \cdot
	L\left(g\otimes\chic{D_1},1/2\right)
	\abs{D_1}^{\mbk + \mathbf{1/2}}
\end{multline}
for every $D_i\in F^\times$ of type $\gamma$, $i = 1,2$.

Here $\id a_i$ is the unique ideal such that $(D_i,\id a_i$) is a fundamental
discriminant.

\end{prop}

This result is a refinement of \cite[Cor.~4.7]{baruch-mao}
(see also \cite[Cor.~2]{waldspu-coeffs}).
Its proof relies heavily on \cite{baruch-mao}, to where we refer the reader for
further details.
The strategy is to use the main result from op.cit. to show that \eqref{eqn:bmw}
holds up to a product of local constants over certain bad primes,
and then showing that these local constants agree.
In order to state that result, we need to introduce the following ingredients.

\medskip

From here on we let $\varphi$ denote the automorphic form on $\PGL_2(\FA)$
corresponding to $g$ (which should not be confused with the quaternionic forms
considered on previous sections).
Let $W_\varphi$ denote the Whittaker function on $\PGL_2\left(\FA\right)$
relative to the character $\psi$, and for every place $v$ denote by $\mathcal
L_v$ a nonzero Whittaker functional on $V_{\pi_v}$ with respect to $\psi_v$.

For $\tphi \in \tpi$ and $D \in F^\times$ we let
$\widetilde{W}_{\tphi}^D$ denote the Whittaker function on
$\widetilde{\SL_2}\left(\FA\right)$ relative to the character $\psi^D$,
given by
\begin{equation}\label{eqn:whitt}
	\widetilde{W}_{\tphi}^D(g)
	=
	\int_{F_\mbA/F}
	\tphi\bigl(\smat{1 & x \\ 0 & 1} g\bigr) \,
	\psi^D(-x) \,
	dx.
\end{equation}

Assume that $D$ is of type $\gamma$.
As in Sect.~\ref{sect:shimura},
letting $\epsilon(v) = \gamma(v)^{\val_v(\id N)}$
we have that $D \in F^\epsilon(\pi)$; hence by
Proposition~\ref{prop:shimura_map2}
there exists a nonzero Whittaker functional $\tlv D$ on $V_{\tpi_v}$
with respect to $\psi_v^D$.
Denote by $\normD[D][\cdot]$ the norm on $V_{\tpi_v}$ induced by $\tlv D$
as in \cite[(2.9)]{baruch-mao}.

Finally, if $v \in \mbf$ and $D \in \OO[v]$ we say that $D$ is a
\emph{fundamental discriminant} if the discriminant of
$F_v(\sqrt D)/F_v$ equals $D \OO[v]$.
Note that if $D \in F^\times$ and $(D,\id a)$ is a fundamental discriminant,
then $D$ is a fundamental discriminant at $v$ for every $v \in \mbf$ such that
$\id a \OO[v] = \OO[v]$.

\medskip

\begin{thm}[{\cite[Thm. 4.3]{baruch-mao}}]
\label{thm:bm}
Let $\tphi \in \tpi$.
Let $D \in F^\times$ of type $\gamma$.
Then for every finite set of places $\Sigma$ containing
$\SN$ and the even places
\begin{multline}\label{eqn:bm}
	\abs{\widetilde W^D_{\tphi}(e)}^2
	\Abs{\varphi}^2
	\prod_{v \in \Sigma}
	L\left(\pi_v\otimes\chic[v]{D},1/2\right)
	\abs{ \mathcal L_v(\varphi_v) }^2
	\normD[D][\tphi_v]^2
	\abs{D}_v
	= \\
	L\left(\pi\otimes\chic{D},1/2\right)
	\Abs{\tphi}^2
	\abs{W_\varphi(e)}^2
	\prod_{v \in \Sigma}
	\abs{ \tlv D\left(\tphi_v\right) }^2
	\Abs{ \varphi_v}_v^2,
\end{multline}
as long as $D$ is a fundamental discriminant outside $\Sigma$.
\end{thm}

\begin{rmk}
Since $D$ is of type $\gamma$, in particular $D \in F^\epsilon(\pi)$.
By propositions \ref{prop:shimura_map} and \ref{prop:shimura_map2} we have
that $\tpie \simeq \tpi$ and $S_\psi(\tpie) = \pi$.
This shows that the hypotheses of the original statement of \cite[Thm.
4.3]{baruch-mao} are satisfied.
We point out that, though Baruch and Mao state their formula for
square-free integers $D$, they only use this fact outside $\Sigma$; see note 
at the bottom of \cite[p.~357]{baruch-mao}.
\end{rmk}

From here on, as in Sect.~\ref{sect:shimura}, we let $\tphi \in \tpi$ be the
vector corresponding to $f$.
Let $D \in F_v^\times$, and assume that
there exists a nonzero Whittaker functional $\tlv D$ on $V_{\tpi_v}$
with respect to $\psi_v^D$.
Then we let
\begin{align*}
	\mcE_v(D) &=
	\frac
	{
	e_v(D)
	}
	{
	\abs{D}_v
	L\left(\pi_v \otimes\chic[v]{D},1/2\right)
	}
	\cdot
	\frac
	{
	\abs{ \tlv{D}(\tphi_v) }^2
	}
	{
	\normD^2
	}
	,
	\qquad \text{where} \\
	e_v(D) &=
	\begin{cases}
		e^{2\pi \abs{D}_v} 
        \abs{D}_v^{1/2-k_v} & \text{if $v \in \mba$,} \\
        2 & \text{if $v \mid \cond D$ and $v \mid \id N$,} \\
		1 & \text{otherwise}.
	\end{cases}
\end{align*}
Note that $\mcE_v(D)$ does not depend on the choice of the Whittaker functional
$\tlv{D}$.
It also remains unchanged if $\tphi_v$ is multiplied by a nonzero scalar.
If a Whittaker functional does not exist, we let $\mcE_v(D) = 0$.

\begin{lemma}\label{lem:ednr}
	Let $v \in \mbf$ be such that $\val_v(\id N) = 0$.
	Then there exists a constant $c_v = c_v(\varphi_v)$ such that
	$\mcE_v(D) = c_v$ for every fundamental discriminant $D \in \OO[v]$.
\end{lemma}

\begin{proof}
Since $v \nmid \id N$ we have that
$\pi_v \simeq \pi_v(\mu,\mu^{-1})$, with 
$\mu$ an unramified character of $F_v^\times$ given by
$\mu(x) = \abs{x}_v^s$, with $s \in i \R$.
Let $q = \abs{\OO[v] / \piv \OO[v]}$.
Then by \cite[(200)]{godement_JL} we have that
\begin{equation}
	\label{eqn:ltwist}
	L\left(\pi_v\otimes \chic[v]{D},1/2\right)
	= 
	\abs{1-\kro Dv q^{-1/2-s}}^{-2}.
\end{equation}

By Proposition \ref{prop:tpiv} we have that $\tpi_v \simeq \tpi_v(\mu)$.
Then we can consider the Whittaker functional $\tlv D$ given in
\cite[(8.3)]{baruch-mao}.
In particular using \cite[(8.5)]{baruch-mao} we get that 
$\abs{D}_v \normD^2$ does not depend on $D$.

By the first claim of Proposition~\ref{prop:kohnen} we have that
$\lambda(D,\OO;f) = 0$ unless there exists $\xi \in \OO$ such that $D \equiv
\xi^2 \pmod{4\OO}$.
Then \cite[Prop.~13.3]{hiraga-ikeda}
implies that $\tphi_v$ is fixed by the idempotent
operator $E_v^{K}$ acting on $V_{\tpi_v}$.
Hence we can assume that $\tphi_v$ equals the vector $f_K^+$ appearing in
\cite[Prop.~4.6]{hiraga-ikeda}.

Finally, we can compute $\tlv D(\tphi_v)$ using \cite[Prop.~4.3]{hiraga-ikeda},
since the Whittaker functional they consider in that proposition agrees with
$\tlv D$.
More precisely, being $D$ fundamental the integer $\cond D$
defined in \cite[Def.~2.2]{hiraga-ikeda} (which should not be confused with
our notation for the conductor) equals $0$.
This implies that the factor $\Psi(D,q^{-s})$ appearing in
\cite[Prop.~4.3]{hiraga-ikeda} formula equals $1$, so we get that
\[
	\tlv D(\tphi_v)
	=
	\left(1-q^{-1-2s}\right)
	\left(1-\kro Dv q^{-1/2-s}\right)^{\!-1}
	\!\!\!,
\]
which together with \eqref{eqn:ltwist} completes the proof.
\end{proof}

We remark that though Hiraga and Ikeda consider the case of trivial level for
stating their main results, the local auxiliary results quoted above are valid
since they are used in the case $v \nmid \id N$.
We also remark that in the case when $v$ is odd the result above follows 
by \cite[Prop.~8.1]{baruch-mao};
furthermore, the case when $F = \Q$ and $v = 2$
can be proved using \cite[Prop.~9.5]{baruch-mao}.
The results from \cite{hiraga-ikeda} allow us to give a general proof.

\begin{lemma}\label{lem:edspecial}
	Let $v \in \mbf$ be such that $\val_v(\id N) = 1$.
	Assume that $v$ is odd.
	Then there exists a constant $c_v = c_v(\varphi_v)$ such that
	$\mcE_v(D) = c_v$ for every fundamental discriminant $D \in \OO[v]$
	such that $\kro Dv \neq - \AL[v]{g}$.
\end{lemma}

\begin{proof}
	By Proposition \ref{prop:tpiv} we have that
	$\tpi_v \simeq \widetilde\sigma_v({\mu})$ with
	$\mu = \abs{\,\cdot\,}^{1/2}_v \chi_v^{\tau}$ and 
	$\tau \in \OOx[v]$ such that $\kro{\tau}v = -\AL[v]{g}$.
	In particular $D \tau \notin (\OOx[v])^2$.
	Moreover,
	Proposition \ref{prop:modularity} (a)
	implies that $\tphi_v$ is fixed by the Iwahori subgroup.
	Then the result follows by applying \cite[Prop.~8.4]{baruch-mao}.
\end{proof}

\begin{rmk}
	\label{rmk:even}
	We expect this result to hold for even $v$ as well.
	This is ongoing work, and can be proved generalizing to even primes the
	framework given in \cite[Sect.~8]{su} for studying Steinberg
	representations.
\end{rmk}

\begin{lemma}\label{lem:edunit}
	Let $v \in \mbf$.
	Let $D_1,D_2 \in F_v^\times$ be such that there exists $t_v \in \OOx[v]$
	with $D_1 = D_2\, t_v^2$.
	Then $\mcE_v(D_1) = \mcE_v(D_2)$.
\end{lemma}

\begin{proof}
	Let $t_v \in F_v^\times$ and denote $\utv t = \smat{t_v & 0 \\ 0 &
	t_v^{-1}}$.
	If $D_1 = D_2 \, t_v^2$, given a nonzero Whittaker functional $\tlv{D_2}$
	we obtain a nonzero Whittaker functional $\tlv{D_1}$ by letting
	\begin{equation}\label{eqn:whittaker!}
		\tlv{D_1}\left(\widetilde\alpha_v\right)
		=
		\tlv{D_2}\left(\utv t \cdot \widetilde\alpha_v\right)
		.
	\end{equation}
	In particular, the induced norms satisfy that
	\begin{equation}\label{eqn:localnorm}
		\normD[D_1][\widetilde\alpha_v]^2
		=
		\normD[D_2][\utv{t} \cdot \widetilde\alpha_v]^2
		=
		\normD[D_2][\widetilde\alpha_v]^2
		.
	\end{equation}

	Assume now that $t_v \in \OOx[v]$.
	Then by \cite[p.~386]{waldspu-coeffs}, which can be used by
	Proposition \ref{prop:modularity} (a)
	and can be used over
	other fields than $\Q$ due to its local nature, there exists a root of
	unity $\delta_v$ such that
	$\utv{t} \cdot \tphi_v = \delta_v\,\tphi_v$.
	This together with \eqref{eqn:whittaker!} and \eqref{eqn:localnorm}
	completes the proof.
\end{proof}

\begin{prop}\label{prop:edf}
	Let $v \in \mbf$.
	Let $D_1,D_2 \in \OO[v]$ be fundamental discriminants of type $\gamma$ at $v$.
	Then $\mcE_v(D_1) = \mcE_v(D_2)$.
\end{prop}

\begin{proof}
	This is immediate from Lemmas \ref{lem:ednr}, \ref{lem:edspecial} and
	\ref{lem:edunit}.
\end{proof}

\begin{prop}\label{prop:eda}
	Let $v \in \mba$.
	Let $D_1,D_2 \in F_v^\times$ be of type $\gamma$ at $v$.
	Then $\mcE_v(D_1) = \mcE_v(D_2)$.
\end{prop}

\begin{proof}
	By
	Proposition \ref{prop:modularity} (a)
	we have that $\tphi_v$ is a vector of minimal weight in $\tpi_v$.
	Since $D_1 = D_2 \,t^2$ with $t \in F_v^\times$ we have
	$L(\pi_v \otimes \chic[v]{D_1},1/2) = L(\pi_v \otimes \chic[v]{D_2},1/2)$.
	\par
    In the case $\gamma(v) = 1$ the result follows by applying
	\cite[Prop.~8.8]{baruch-mao} to $D = D_i/2$.
	In the case when $\gamma(v) = -1$, i.e., when $\tphi_v$ is
	antiholomorphic, the result of Baruch and Mao also holds if we
	change $D$ by $-D$ in their formula for $e(\tphi_v,\psi^D)$
	as can be seen using the formula for the Whittaker model given in
	\cite[p.~24]{waldspu_corresp}.
\end{proof}

\begin{proof}[Proof of Proposition~\ref{prop:bmw}]

Let $D \in F^\times$ be of type $\gamma$, and let $\id a \in \IF$.
By our choice of $\psi$ the $(D,\id a)$-th Fourier coefficient of $f$
is given by
\begin{equation}\label{eqn:coeffsW}
	\lambda(D,\id a;f) =
	\chi_\mbf(t) \,
	\idn a^{1/2} \,
	e^{\pi \sum_{v \in \mba} \abs{D}_v} \,
	\widetilde W^D_{\ut t \cdot \tphi}(e)
\end{equation}
where $t \in \hatx F$ is such that $t \OO = \id a$ 
and $\ut t = \smat{t & 0 \\ 0 & t^{-1}}$ 
(see \eqref{eqn:fourier} and \eqref{eqn:whitt}).

We apply Theorem~\ref{thm:bm} to $\ut t \cdot \tphi$.
Since the norm on $V_\tpi$ is $\widetilde
\SL_2\left(\FA\right)$-equivariant,
using \eqref{eqn:whittaker!} and \eqref{eqn:localnorm} 
we can rewrite \eqref{eqn:bm} as

\begin{multline}\label{eqn:bm2}
	\abs{\widetilde W^D_{\ut t \cdot \tphi}(e)}^2
	\Abs{\varphi}^2
	\prod_{v \in \Sigma}
	L\left(\pi_v\otimes\chic[v]{Dt_v^2},1/2\right)
	\abs{ \mathcal L_v(\varphi_v) }^2
	\normD[Dt_v^2]^2
	\abs{D}_v
	= \\
	L\left(\pi\otimes\chic{D},1/2\right)
	\Abs{\tphi}^2
	\abs{W_\varphi(e)}^2
	\prod_{v \in \Sigma}
	\abs{ \tlv{Dt_v^2}\left(\tphi_v\right) }^2
	\Abs{ \varphi_v}_v^2.
\end{multline}

Now let $(D_i, \id a_i)$ be fundamental discriminants with $D_i$ of type
$\gamma$, and let $t_i \in \hatx F$ be such that $t_i \OO = \id a_i$.
We let $\Sigma$ be a finite set of places containing $\SN$ and the even places,
and satisfying that if $v \notin \Sigma$ then $t_{i,v} \in \OOx[v]$
(so that $D_i$ is fundamental at $v$).
Then \eqref{eqn:coeffsW} together with \eqref{eqn:bm2} applied to $D = D_i$ and
$t = t_i$ imply that
\begin{multline}
	\label{eqn:bmw_aux}
	2^{\omega(D_1, \id N)}
	\idn[2]{a} \,
	\abs{\lambda\left(D_1, \id a_1;f\right)}^2 \,
	L\left(g\otimes\chic{D_2},1/2\right) \,
	\prod_{v \in \Sigma} \mcE_v\left(D_2\, t_{2,v}^2\right) \,
	\abs{t_{2,v}}_v^2
	\\
	=
	2^{\omega(D_2, \id N)}
	\idn[1]{a} \,
	\abs{ \lambda\left(D_2, \id a_2;f\right) }^2 \,
	L\left(g\otimes\chic{D_1},1/2\right)  \,
	\prod_{v \in \Sigma} \mcE_v\left(D_1\, t_{1,v}^2\right) \,
	\abs{t_{1,v}}_v^2.
\end{multline}
By Propositions \ref{prop:edf} and \ref{prop:eda} we have that
$\mcE_v\left(D_1\, t_{1,v}^2\right) = \mcE_v\left(D_2\, t_{2,v}^2\right)$ for
every $v \in \Sigma$.
Then \eqref{eqn:bmw} follows from \eqref{eqn:bmw_aux}, since
$
	\prod_{v \in \Sigma} \abs{t_{i,v}}_v^2
	=
	\mathcal N(\mathfrak a_i)^{-2}
	.
$
\end{proof}


\section{The rational case}
\label{sect:Q}

Assume throughout this section that $F = \Q$.

\medskip

Given a positive integer $N$,
denote $\Sigma_N = \set{v}{v\mid N} \cup \{\infty\}$.
Let $g \in S_{2+2k}(N)$ be a normalized Hecke newform,
and let $\gamma : \Sigma_N
\to \{\pm1\}$ be such that 
\begin{equation}\label{eqn:signo_de_gammaQ}
	\prod_{v \mid N}
	\AL[v]{g}\,
	\gamma(v)^{\val_v(N)}
	=
	(-1)^{k+1} \gamma(\infty)
	\,
\end{equation}
(see \eqref{eqn:signo_de_gamma}).
Then Theorem~\ref{thm:main}, which generalizes the result of Baruch and Mao for
square-free levels mentioned in the Introduction (\cite[Thm. 10.1]{baruch-mao}),
can be stated in the following form:

\begin{thm}\label{thm:Q}
Assume that $N$ is not a square.
There exists a nonzero (holomorphic) cuspidal form $f\sg$ of weight $3/2+k$
whose Fourier coefficients $\lambda(\vert D \vert;f\sg)$ are effectively
computable and satisfy
\begin{equation}\label{eqn:formulaQ}
	L\left(1/2,g\otimes \chic{D}\right)
	=
	2^{\omega(D,N)} \, 
	c_{g,\gamma} \,
	\ll{g,g} \,
	\frac{1}{\abs{D}^{k+1/2}} \,
	\frac{\abs{ \lambda(\vert D \vert;f\sg) }^2}{\ll{f\sg,f\sg}}
\end{equation}
for every fundamental discriminant $D \in \Z$ such that
\begin{equation*}
	\kro Dv =
	\begin{cases}
		\gamma(v) \text{ or } 0 &
		\text{when $\val_v(N) = 1$, $v$ is odd and $\gamma(v) = \AL[v]{g}$,} \\
		\gamma(v) & \text{otherwise}.
	\end{cases}
\end{equation*}

Here $c_{g,\gamma}$ is given in $\eqref{eqn:c_gamma}$
and $\omega(D,N)$ denotes the number of primes dividing both $N$
and $D$.

The level of $f\sg$ is $4NN'N''$, where $N'$ is the
	product of the primes in the set $\set{v\parallel
N}{\gamma(v)\neq\AL[v]{g}}$ and
\begin{equation*}\label{eqn:N''}
	N'' =
	\begin{cases}
		4, & 2 \parallel N \text{ and } \AL[2]{g} \neq \gamma(2),\\
		4, & 4 \parallel N \text{ and } \AL[2]{g} \neq 1 ,\\
		2, & 8 \parallel N \text{ and } \AL[2]{g} \neq \gamma(2),\\
		1, & \text{otherwise}.
	\end{cases}
\end{equation*}
\par
\end{thm}

Note that the form in Theorem~\ref{thm:main} can be mapped 
to a holomorphic $f\sg$ as in the statement above by applying the isomorphism in
Proposition~\ref{prop:hilbert_holomorphic} with $u=-\gamma(\infty)$.

\medskip

Even though our construction is more general, for simplicity we choose a prime
$p$ such that $\val_p(N)$ is odd, we let $B$ be the quaternion algebra over $\Q$
ramified at $\infty$ and at $p$, and we consider a Pizer-Eichler order $R
\subseteq B$ of discriminant $N$.
The modular form $f\sg$ in Theorem~\ref{thm:Q} is given by $\tlu(\varphi)$,
where $\varphi$ is a quaternionic modular form on $B$ in Jacquet-Langlands
correspondence with $g$, and $\tlu$ is the map from \eqref{eqn:theta_map}.
Here $l$ and $\chi$ must be chosen according to $\gamma$; this choice, which
determines the weight functions involved in $\tlu$ (see \eqref{eqn:wadelic}),
can be done as follows.

Firstly, $l$ is required to satisfy that $L(1/2,g\otimes\chi^l) \neq 0$.
Furthermore, we require that $\gamma(\infty) \, l <
0$, and that for $q \mid N$
\begin{equation*}
	\kro lq =
	\begin{cases}
		-\gamma(q),& q = p,\\
		\gamma(q),& q \neq p.
	\end{cases}
\end{equation*}
For simplicity, $l$ can be chosen either as $l=1$ or as a prime not
dividing $N$ with $l \equiv 1 \pmod{4}$.
In particular $\cond{l} = (l)$ is square-free and prime to $2N$ as required in
Sect.~\ref{sect:construction}, and at most one weight function of type I is
involved.

A weight function of type II is required for every $q \mid N$ such that
$\gamma(q)^{\val_q(N)} \neq \AL[q]g$.
For such $q$ we let $\chi_q$ be an odd Dirichlet character of
conductor $q$ if $q$ is odd, and we let $\chi_2$ be the odd character of
conductor $4$ if $q=2$.
Then we let $\chi$ be the primitive Dirichlet character 
obtained by multiplying the characters $\chi_q$.
With this choice of $\chi$ we have that the level of $f_\gamma$ equals
$4\lcm(N,\cond{\chi}^2) = 4NN'N''$, with
$N',N''$ as above (see Proposition~\ref{prop:modularity}).

\medskip

When $N = p$ and $k=0$, by \eqref{eqn:signo_de_gammaQ} our theta series involve
a weight function of type II at $p$ if and only if $\gamma(\infty) = 1$.
Hence we recover the weight $3/2$ forms $g_-$ and $g_+$ from \cite{mrvt}
(not to be confused with our form $g$), according to
whether $\gamma(\infty)$ equals $-1$ or $1$.
Furthermore, in this case our result (and our construction) agrees with that of
\cite{mao}.

In the case of composite levels we recover the examples given in
\cite{pt_composite}; furthermore, the choice of $\chi$ given above explains
which are the weight functions of type II needed in the examples from op.cit.,
according to each function $\gamma$.
We remark that those examples were not shown to satisfy \eqref{eqn:formulaQ}
prior to this work.

\medskip

We refer the reader to \cite{mrvt,pt_composite} (and to
Sect.~\ref{sect:example}) for fully detailed examples.
We conclude this section illustrating briefly Theorem~\ref{thm:Q} in a
case of even level, which is not considered in op.cit.

\subsection{A form of even level}

We consider the cuspidal newform $g$ with LMFDB label
\href{https://www.lmfdb.org/ModularForm/GL2/Q/holomorphic/50/2/a/a/}{50.2.a.a}
(see \cite{lmfdb}).
This is the first modular of weight $2$ with rational eigenvalues 
whose level is exactly divisible by $2$ and is not square-free.
Its Atkin-Lehner eigenvalues are given by $\AL[2]{g} = 1, \AL[5]{g} = -1$.

There are four types $\gamma$ satisfying \eqref{eqn:signo_de_gammaQ}.
We focus our attention on the type given by
$\gamma(\infty) =
1
, \gamma(2) =
1
, \gamma(5) =
1
$.
Taking $l=-11$ and a weight function of type II at $5$ we get
\begin{equation*}
	f_{\gamma} =
		q +  q^{4} -  q^{16} - 3 q^{24} - 3 q^{41} + 3 q^{44} -  q^{49} -  q^{64} + 3 q^{76} + 3 q^{81} + 3 q^{89} + 3 q^{96}
	+ O(q^{100}).
\end{equation*}
According to Theorem~\ref{thm:Q} we have that
\begin{equation*}
	L\left(1/2,g\otimes \chic{D}\right)
	=
	2^{\omega(D,N)} \, 
	C_\gamma \,
	\frac{\abs{ \lambda(D;f\sg) }^2}{D^{1/2}}
\end{equation*}
for every positive fundamental discriminant $D$ such that
$\kro D2 = \kro D5 = 1$, where
$C_\gamma =
0.137248\dots\;$.

The form $f\sg$ has even fundamental discriminants in its support.
We expect our results to hold as well for these discriminants,
by allowing $\kro D2$ to be zero when $\val_2(N) = 1$ and $\gamma(2) =
\AL[2]{g}$ in Theorem~\ref{thm:Q}; see Remark~\ref{rmk:even}.
This was numerically verified in this case, as well as in the other examples of
modular forms of even weight considered in \cite{code}.

\begin{rmk}
	The central values $L\left(1/2,g\otimes\chi^{mD}\right)$ for odd
	fundamental discriminants $D$ and $m \in \{-4,8,-8\}$
	could, in principle, be obtained by applying 
	Theorem~\ref{thm:Q} to $g\otimes\chi^{m}$.
	But the latter has level $\lcm(50,m^2)$, which is a square.
	Our result does not apply in this case; a different family of
	orders must be considered
	(see \cite{ptII}, where the case of forms level $p^2$
	is treated).
	Nevertheless, this technique could be used in other cases, e.g. 
	for the form $g$ with label
	\href{https://www.lmfdb.org/ModularForm/GL2/Q/holomorphic/14/2/a/a/}{14.2.a.a}.
\end{rmk}


\section{An example over \texorpdfstring{$\Q(\sqrt3)$}{Qsqrt3}}
\label{sect:example}

Let $F = \Q(a)$, with $a^2 = 3$.
We let $\mba = \{v_1,v_2\}$, with $v_1(a) = -\sqrt 3$ and $v_2(a) = \sqrt 3$.
We have that $h_F = 1$ and $h_F^+ = 2$.
Moreover, every unit in $\OO$ is either totally positive or totally
negative.
In particular our construction will give skew-holomorphic modular forms that
cannot be mapped to holomorphic forms preserving the level
(see Remark~\ref{rmk:uholo}).

\medskip

We consider the cuspidal newform $g$ over $F$ with LMDFB label
\href{https://www.lmfdb.org/ModularForm/GL2/TotallyReal/2.2.12.1/holomorphic/2.2.12.1-47.2-b}{2.2.12.1-47.2-b}
(see \cite{lmfdb}); this is the first Hilbert modular
form over $F$ of weight $\mathbf 2$, prime level, rational Hecke eigenvalues,
and $-$ sign in the functional equation (so that, in particular, $L(1/2,g)=0$).
It has level $\id p = (4a -1)$, with $\idn p = 47$, and its
Atkin-Lehner eigenvalue at $\id p$ is $\varepsilon_g(\id p) = -1$.

Every $D \in F^\times$ with + sign in the functional equation for
$L(s,g\otimes\chic{D})$ is of one of the four types
listed in Table~\ref{tab:types}.
For $D$ not covered in this table
we have trivially $L(1/2,g\otimes\chic{D})=0$.

\begin{table}[ht]
\begin{tabular}{c cccc}
    Type of $D$ & $\sgn v_1(D)$ & $\sgn v_2(D)$ &
    $\kro{D}{\id p}$
    \\[6pt]\hline
    $\gamma_1$ & $+$ & $+$ & $-1$ or $0$ \\
    $\gamma_2$ & $-$ & $-$ & $-1$ or $0$ \\
    $\gamma_3$ & $+$ & $-$ & $1$ \\
    $\gamma_4$ & $-$ & $+$ & $1$ \\
\end{tabular}
\caption{Types with $\varepsilon_{g,\gamma_i} = 1$}
\label{tab:types}
\end{table}

\medskip

We let $B$ be the quaternion algebra over $F$ given by $i^2 = j^2 = -1, ij = -ji
= k$, which is ramified exactly at $\mba$.
We consider the Eichler order $R\subseteq B$ with discriminant $\id p$ given by
\[
	R = \ll{1,\, (-4a + 1)i,\, \tfrac12((a - 38)i + j),\, \tfrac12(a + 34i + k)}. \\
\]
It has class number $4$.
The elements $x_i \in \hatx B$ representing the right $R$-ideal classes have
norms given by
\begin{equation}
	\label{eqn:norms}
	\norm(x_1) = (1), \;
	\norm(x_2) = (6a-25), \;
	\norm(x_3) = (-4a+1), \;
	\norm(x_4) = (2a-1).
\end{equation}
They induce the following ternary lattices:
\begin{align*}
	L_{x_1} & = L_{x_3} =
	\ll{(-4a + 1)i,\, \tfrac12((a - 38)i + j),\, \tfrac12(a + 34i + k)},\\
	L_{x_2} & = L_{x_4} =
	\big\langle
	(6a - 25)i,\,
	\tfrac12((a - 504)i + j),\,
	\tfrac{1}{22}(-11a - (2a+1)(3736i + 4j + k))
			\big\rangle.
\end{align*}
Furthermore, the quaternionic modular form $\varphi$ corresponding to $g$ is
given by
\[
	\varphi
	=
	-\varphi_{x_1}
	+\varphi_{x_2}
	+\varphi_{x_3}
	-\varphi_{x_4}
    ,
\]
where $\varphi_x$ stands for $\varphi_{x,1}$ as given by~\eqref{eqn:qmf}.

\medskip

We now describe each modular form $f_\gamma = \tlu(\varphi)$ as defined by
\eqref{eqn:theta_map}.
Among the Hecke characters $\chi$ of smallest possible conductor satisfying
\ref{hyp:Hchi}, in each case we were able to choose one with $\chi_{\id p}$
quadratic, for convenience of computation;
otherwise Remark~\ref{rmk:oddchi} would have been applied.
Furthermore, in each case the auxiliary parameter $l \in F^\times$ satisfying
\ref{hyp:Hl} is chosen such that $\cond l$ is either trivial or prime, so that
we only need to consider at most one weight function of type I at $\cond {l}$
for computing the weight function $w$ given in \eqref{eqn:wadelic}.
For convenience, we also choose $l$ such that $(l, \OO)$ is fundamental.

\begin{itemize}
\item Let $\gamma_1$ be given by
	$\gamma_1(v_1) =
	1
	, \gamma_1(v_2) =
	1
	, \gamma_1(\id p) =
	-1
	$.
	We have that $\cond{\chi} = \OO$, and
	we take $l =
	-1
	$.
	Then $f_{\gamma_1}$ is holomorphic and has level $4\id p$.

\item Let $\gamma_2$ be given by
	$\gamma_2(v_1) =
	-1
	, \gamma_2(v_2) =
	-1
	, \gamma_2(\id p) =
	-1
	$.
	We have that $\cond{\chi} = \OO$, and
	we take $l =
	5
	$.
	Then $f_{\gamma_2}$ is antiholomorphic at both variables and has level $4\id
	p$.
	It can be mapped to a holomorphic form of the same level using
	Proposition~\ref{prop:hilbert_holomorphic}, letting $u = -1$.
\item Let $\gamma_3$ be given by
	$\gamma_3(v_1) =
	1
	, \gamma_3(v_2) =
	-1
	, \gamma_3(\id p) =
	-1
	$.
	We have that $\cond{\chi} = \id p$, and
	we take $l =
	4a + 5
	$.
	Then $f_{\gamma_3}$ is holomorphic at $v_1$ and antiholomorphic at $v_2$,
	and has level $4\id p^2$.
	It cannot be mapped to a holomorphic form of the same level using
	Proposition~\ref{prop:hilbert_holomorphic}, since there does not exist $u
	\in \OOx$ such that $\sgn u_{v_1} = 1$ and $\sgn u_{v_2} = -1$.
\item Let $\gamma_4$ be given by
	$\gamma_4(v_1) =
	-1
	, \gamma_4(v_2) =
	1
	, \gamma_4(\id p) =
	-1
	$.
	We have that $\cond{\chi} = \id p$, and
	we take $l =
	-4a+5
	$.
	Then $f_{\gamma_3}$ is antiholomorphic at $v_1$ and holomorphic at $v_2$,
	and has level $4\id p^2$.
	Again, it cannot be mapped to a holomorphic form of the same level.
\end{itemize}

\medskip

Finally, using \eqref{eqn:cfourier} we compute the first Fourier coefficients
$\lambda(D) = \lambda(D,\OO;f\ssg)$ with $(D,\OO)$ fundamental.
Since the four lattices $L_{x_i}$ are locally the same at $\id p$ and at
$\cond {l}$ (we chose the representatives for the ideal classes so that this was
true), the functions $w_x$ appearing in \eqref{eqn:cfourier} can be computed in
terms of $w$ using \eqref{eqn:wxw} and \eqref{eqn:norms}.
Moreover, we compute the corresponding central values $L_D = L(1/2,g\otimes
\chi^D)$ using Theorem~\ref{thm:formula}; see tables
\ref{tab:g1},
\ref{tab:g2},
\ref{tab:g3} and
\ref{tab:g4}.

The entries in these tables are ordered according to the norms $\mathcal N(D)$
of the discriminants.
Only one element of each coset $D \cdot (\OOx)^2$ was included, since for $u \in
\OOx$ we have that $\lambda(u^2 D)$ is related to $\lambda(D)$ by
\cite[(3.14b)]{shim-hh}.

\begin{table}[ht]
\begin{tabular}{cc|cc || cc|cc}
	$\mathcal N(D) $ & $D$ & $\lambda(D)$ & $L_D$
		&
	$\mathcal N(D) $ & $D$ & $\lambda(D)$ & $L_D$
	\\
\hline
$25$ & $5$ & $2$ & $4.574$ & $148$ & $-6a + 16$ & $2$ & $1.880$ \\ $33$ & $4a +
9$ & $-2$ & $3.981$ & $169$ & $13$ & $0$ & $0.000$ \\ $52$ & $-2a + 8$ & $-2$ &
$3.172$ & $193$ & $-12a + 25$ & $0$ & $0.000$ \\ $73$ & $-4a + 11$ & $0$ &
$0.000$ & $208$ & $-4a + 16$ & $-2$ & $1.586$ \\ $96$ & $4a + 12$ & $2$ &
$2.334$ & $244$ & $-2a + 16$ & $2$ & $1.464$ \\ $96$ & $-4a + 12$ & $-2$ &
$2.334$ & $244$ & $2a + 16$ & $-2$ & $1.464$ \\ $97$ & $8a + 17$ & $2$ & $2.322$
		& $249$ & $8a + 21$ & $-2$ & $1.449$ \\ $97$ & $-8a + 17$ & $-2$ &
$2.322$ & $249$ & $-8a + 21$ & $2$ & $1.449$ \\ $121$ & $11$ & $2$ & $2.079$ &
$276$ & $-10a + 24$ & $0$ & $0.000$ \\ $148$ & $6a + 16$ & $2$ & $1.880$ & $313$
	  & $4a + 19$ & $2$ & $1.293$
\\
\end{tabular}
\vspace{-1ex}
\caption{Fourier coefficients for $f_{\gamma_1}$ and their corresponding twisted
central values}
\label{tab:g1}
\end{table}

\vspace{-1ex}

\begin{table}[ht]
\begin{tabular}{cc|cc || cc|cc}
	$\mathcal N(D) $ & $D$ & $\lambda(D)$ & $L_D$
		&
	$\mathcal N(D) $ & $D$ & $\lambda(D)$ & $L_D$
	\\
\hline
$1$ & $-1$ & $2$ & $1.400$ & $241$ & $4a - 17$ & $-8$ & $1.443$ \\ $16$ & $4a -
8$ & $-4$ & $1.400$ & $241$ & $-4a - 17$ & $4$ & $0.361$ \\ $33$ & $4a - 9$ &
$4$ & $0.975$ & $276$ & $-10a - 24$ & $0$ & $0.000$ \\ $49$ & $-7$ & $-4$ &
$0.800$ & $289$ & $-17$ & $-4$ & $0.329$ \\ $52$ & $-2a - 8$ & $0$ & $0.000$ &
$313$ & $4a - 19$ & $8$ & $1.266$ \\ $73$ & $-4a - 11$ & $4$ & $0.655$ & $321$ &
$-16a - 33$ & $-4$ & $0.313$ \\ $177$ & $4a - 15$ & $12$ & $3.788$ & $337$ & $8a
- 23$ & $4$ & $0.305$ \\ $177$ & $-4a - 15$ & $4$ & $0.421$ & $337$ & $-8a - 23$
	  & $0$ & $0.000$ \\ $193$ & $-12a - 25$ & $-4$ & $0.403$ & $352$ & $-4a -
20$ & $-8$ & $1.194$ \\ $208$ & $-4a - 16$ & $8$ & $1.553$ & $352$ & $12a - 28$
	& $-16$ & $4.776$
\\
\end{tabular}
\vspace{-1ex}
\caption{Fourier coefficients for $f_{\gamma_2}$ and their corresponding twisted
central values}
\label{tab:g2}
\end{table}

\begin{table}[ht]
\begin{tabular}{cc|cc || cc|cc}
	$\mathcal N(D) $ & $D$ & $\lambda(D)$ & $L_D$
		&
	$\mathcal N(D) $ & $D$ & $\lambda(D)$ & $L_D$
	\\
\hline
$-23$ & $-4a + 5$ & $-4$ & $2.807$ & $-143$ & $-16a + 25$ & $-4$ & $1.126$ \\
$-32$ & $-4a + 4$ & $4$ & $2.380$ & $-167$ & $-8a + 5$ & $0$ & $0.000$ \\ $-32$
	  & $-12a + 20$ & $-4$ & $2.380$ & $-176$ & $-20a + 32$ & $8$ & $4.060$ \\
$-39$ & $-4a + 3$ & $-4$ & $2.156$ & $-183$ & $-8a - 3$ & $-4$ & $0.995$ \\
$-44$ & $-10a + 16$ & $0$ & $0.000$ & $-183$ & $-8a + 3$ & $-4$ & $0.995$ \\
$-71$ & $-12a + 19$ & $4$ & $1.598$ & $-300$ & $-10a$ & $8$ & $3.109$ \\ $-71$ &
$-8a + 11$ & $0$ & $0.000$ & $-311$ & $-12a + 11$ & $-8$ & $3.054$ \\ $-111$ &
$-8a + 9$ & $4$ & $1.278$ & $-327$ & $-16a + 21$ & $4$ & $0.745$ \\ $-111$ &
$-20a + 33$ & $-4$ & $1.278$ & $-332$ & $-14a + 16$ & $8$ & $2.956$ \\ $-143$ &
$-12a + 17$ & $8$ & $4.504$ & $-332$ & $-34a + 56$ & $-8$ & $2.956$
\\
\end{tabular}
\caption{Fourier coefficients for $f_{\gamma_3}$ and their corresponding twisted
central values}
\label{tab:g3}
\end{table}

\begin{table}[ht]
\begin{tabular}{cc|cc || cc|cc}
	$\mathcal N(D) $ & $D$ & $\lambda(D)$ & $L_D$
		&
	$\mathcal N(D) $ & $D$ & $\lambda(D)$ & $L_D$
	\\
\hline
$-12$ & $2a$ & $4$ & $2.747$ & $-191$ & $8a - 1$ & $12$ & $6.197$ \\ $-23$ & $4a
+ 5$ & $4$ & $1.984$ & $-191$ & $8a + 1$ & $-4$ & $0.689$ \\ $-39$ & $4a + 3$ &
$0$ & $0.000$ & $-236$ & $10a + 8$ & $4$ & $0.619$ \\ $-44$ & $10a + 16$ & $4$ &
$1.435$ & $-236$ & $38a + 64$ & $-4$ & $0.619$ \\ $-47$ & $4a + 1$ & $-4$ &
$1.388$ & $-239$ & $16a + 23$ & $-8$ & $2.462$ \\ $-48$ & $4a$ & $-4$ & $1.374$
		& $-239$ & $20a + 31$ & $0$ & $0.000$ \\ $-143$ & $8a + 7$ & $-4$ &
$0.796$ & $-263$ & $12a + 13$ & $-4$ & $0.587$ \\ $-143$ & $28a + 47$ & $-8$ &
$3.183$ & $-263$ & $32a + 53$ & $-8$ & $2.347$ \\ $-167$ & $8a + 5$ & $-8$ &
$2.945$ & $-311$ & $12a + 11$ & $-4$ & $0.540$ \\ $-176$ & $20a + 32$ & $-4$ &
$0.717$ & $-327$ & $16a + 21$ & $-8$ & $2.105$
\\
\end{tabular}
\caption{Fourier coefficients for $f_{\gamma_4}$ and their corresponding twisted
central values}
\label{tab:g4}
\end{table}

\end{document}